\documentclass[final]{siamltex}

\usepackage{amsmath}
\usepackage{amsfonts}
\usepackage[english]{babel} 

 \usepackage{amssymb}
\usepackage{graphicx}        % standard LaTeX graphics tool
 \usepackage{cite}                            % when including figure files
\usepackage{subfigure}
\usepackage{epstopdf}
\usepackage{epsfig}
\usepackage[table]{xcolor}
\usepackage{subfloat}
\usepackage{float}
\usepackage[thinlines]{easytable}
\usepackage{array}
\allowdisplaybreaks
\numberwithin{equation}{section}

\newtheorem{remark}{Remark}[section]
\newtheorem{algorithm}{Algorithm}[section]

\def\RR{\mathbb R}
\def\EE{\mathbb E}
\def\llambda{\Lambda}

\def\DH{{\mathcal{D}}}
\def\LL{L^1_2(\DH\times\RR^{d_v})}

\def\LH{L^1_2(\RR^{d_v})}

\def\LLBi{L^1_2(\DH\times\RR^{d_v};L^2(\Omega))}
\def\LHBi{L^1_2(\RR^{d_v};L^2(\Omega))}

\def\w{v}

\def\var{{\rm Var}}
\def\cov{{\rm Cov}}
\def\C{{\rm C}}
\def\be{\begin{equation}}
\def\ee{\end{equation}}
\def\bea{\begin{eqnarray}}
\def\eea{\end{eqnarray}}

\title{Multi-scale variance reduction methods based on multiple control variates for kinetic equations with uncertainties\footnote{\tiny
The research that led to the present article was partially supported by the INdAM-GNCS 2018 project \emph{Numerical methods for multi-scale control problems and applications}.}}
\author{Giacomo Dimarco\footnote{\tiny Department of Mathematics \& Computer Science, University of Ferrara, Via Machiavelli 30, Ferrara, 44121, Italy (giacomo.dimarco{@}unife.it).} and Lorenzo Pareschi\footnote{\tiny Department of Mathematics \& Computer Science, University of Ferrara, Via Machiavelli 30, Ferrara, 44121, Italy (lorenzo.pareschi{@}unife.it).}} 

\begin{document}
\maketitle
\begin{abstract} The development of efficient numerical methods for kinetic equations with stochastic parameters is a challenge due to the high dimensionality of the problem. Recently we introduced a multiscale control variate strategy which is capable to accelerate considerably the slow convergence of standard Monte Carlo methods for uncertainty quantification. Here we generalize this class of methods to the case of multiple control variates. We show that the additional degrees of freedom can be used to improve further the variance reduction properties of multiscale control variate methods.\end{abstract}
\begin{keywords}
Uncertainty quantification, kinetic equations, Monte Carlo methods, multiple control variates, multi-scale methods, multi-fidelity methods
\end{keywords}
 
\begin{AMS}
	76P05, 65M75, 65C05
\end{AMS}

\section{Introduction}
Hyperbolic and kinetic equations with random inputs have attracted a lot of attention in the recent years \cite{DPL, JinPareschi, LeMK, NTW, PDL, PIN_book, ZLX}. Most of the literature on kinetic equations is based on the use of Stochastic-Galerkin methods based on on generalized Polynomial Chaos \cite{albi2015MPE, HJ,ZJ, LJ, JinPareschi,JPH,RHS} and only recently these problems have been analyzed in the framework of statistical sampling methods based on Monte Carlo (MC) techniques \cite{Caflisch, Giles, PT2, DPms, JinPareschi, DPZ, HH}.

When dealing with kinetic equations, MC sampling methods present several advantages since they afford simple integration of existing fast deterministic numerical solvers and parallelization techniques which are essential to reduce the computational complexity of many kinetic equations \cite{DP15, DLNR, DPR}. In addition MC methods are effective when the probability distribution of the random inputs is not known analytically or lacks of regularity when approaches based on stochastic orthogonal polynomials may be impossible to use or may produce poor results \cite{MSS, Xu}.

Recently in \cite{DPms} we introduced a control variate technique which takes advantage of the multi-scale nature of the kinetic equation which is capable to strongly accelerate the slow convergence of MC methods. In this manuscript we generalize this class of multi-scale control variate methods (MSCV) to the case of multiple-control variates. In particular we will show how the use of additional control variates functions permit to further reduce the statistical variance of the methods. We refer to \cite{PGW, PWG} for related approaches in the general framework of multifidelity models.  

From a mathematical viewpoint, we consider kinetic equations of the general form \cite{Cer, DPms}
\be
\partial_t f+ {\w}\cdot \nabla_x f = \frac1{\varepsilon}Q(f,f),
\label{eq:FP_general}
\ee
where $f=f({z},x,{\w},t)$, $t\ge 0$, $x\in\DH \subseteq \RR^{d_x}$, ${\w}\in\RR^{d_{\w}}$, $d_x,d_{\w}\ge 1$, and ${z}\in\Omega\subseteq\RR^{d_{z}}$, $d_{z} \geq 1$, is a {random variable}. In \eqref{eq:FP_general} the parameter $\varepsilon > 0$ is the Knudsen number and the particular structure of the interaction term $Q(f,f)$ depends on the kinetic model considered.

The most famous example is represented by the nonlinear Boltzmann equation of rarefied gas dynamics 
\be
Q(f,f)=\int_{S^{d_{\w}-1}\times\RR^{d_{\w}}} B({\w},{\w}_*,\omega,{z}) (f'f'_*-f f_*)\,d{\w}_*\,d\omega
\label{eq:Boltzmann}
\ee
where $d_v \geq 2$ and 
\be
v'=\frac12(v+v_*)+\frac12(|v-v_*|\omega),\quad v_*'=\frac12(v+v_*)-\frac12(|v-v_*|\omega).
\ee

Before going into the details of the presentation we introduce some notations and assumptions that will be used in the sequel.
If $z\in \Omega$ is distributed as $p(z)$ we denote the expected value by
\be
\EE[f](x, {\w}, t) = \int_{\Omega} f(z,x, {\w}, t)p(z)\,dz,
\ee
and we introduce the following $L^p$-norm with polynomial weight \cite{PP, RSS}
\be
\| f(z,\cdot,t)\|^p_{L^p_s(\DH\times\RR^{d_v})} = \int_{\DH\times\RR^{d_v}} |f(z,x,v,t)|^p (1+|v|)^s\,dv\,dx.
\ee
Next, for a random variable $Z$ taking values in ${L^p_s(\DH\times\RR^{d_v})}$, we define 
\be
\|Z\|_{{L^p_s(\DH\times\RR^{d_v};L^2(\Omega))}}=\|\EE[Z^2]^{1/2}\|_{{L^p_s(\DH\times\RR^{d_v})}}.
\label{eq:norm1}
\ee
The above norm, if $p\neq 2$, differs from   
\be
\|Z\|_{L^2(\Omega;{L^p_s(\DH\times\RR^{d_v}))}}=\EE\left[\|Z\|^2_{L^p_s(\DH\times\RR^{d_v})}\right]^{1/2},
\label{eq:norm2}
\ee
used for example in \cite{MSS}. Note that by Jensen inequality \cite{Rudin} we have
\be
\|Z\|_{{L^p_s(\DH\times\RR^{d_v};L^2(\Omega))}} \leq \|Z\|_{L^2(\Omega;{L^p_s(\DH\times\RR^{d_v}))}}.
\ee
To avoid unnecessary difficulties in the sequel we consider norm \eqref{eq:norm1} for $p=1$. The same results hold true for $p=2$ (the two norms coincides) whereas their extension to norm \eqref{eq:norm2} for $p=1$ typically requires $Z$ to be compactly supported. We refer to \cite{MSS, DPms} and \cite[Chapter 7]{JinPareschi} for further details.

Finally, we assume that the deterministic solver for (\ref{eq:FP_general}), if the initial data $f_0$ is sufficiently regular, satisfies the estimate (see \cite{RSS, DP15, Son, DPms, MP})
\be
\|f(\cdot,t^n)-f_{\Delta x,\Delta {\w}}^n\|_{{\LL}} \leq C \left(\Delta x^p+\Delta {\w}^q \right),
\label{eq:det}
\ee
with $C$ a positive constant which depends on time and on the initial data, and $f_{\Delta x,\Delta {\w}}^n$ the computed approximation of the deterministic solution $f(x,v,t)$ at time $t^n$ on the mesh $\Delta x$ and $\Delta v$. Here the positive integers $p$ and $q$ characterize the accuracy of the discretizations in the phase-space and for simplicity, we ignored the errors due to the time discretization and to the truncation of the velocity domain in the deterministic solver \cite{DP15}. 
We emphasize that the estimates here presented are purposely of a general nature to illustrate the characteristics of the method, the application to specific kinetic models can clearly lead to more targeted estimates but it is outside the objectives of the present work.  

The remaining sections can be summarized as follows. In the next Section we survey the MSCV methods introduced in \cite{DPms}. Then, in Section 3 we extend these multi-scale approaches to the case of multiple control variates. First we introduce a standard multiple control variate approach and then consider the construction of recursive multiple control variate methods based on a hierarchical structure. A particular attention is devoted to the case of two control variates. 
 Section 4 is devoted to present several numerical examples confirming the capability of the multiple control variate approach to provide a strong variance reduction with respect to standard Monte Carlo and to further improve the methods based on a single control variate. Some conclusions are reported at the end of the manuscript.  

\section{Multi-scale control variate methods}
\label{sec:MSCV}
In this section we survey the MSCV approach recently introduced in \cite{DPms}. The main idea of the method is to reduce the variance of standard Monte Carlo estimators using as control variate the solution of a simpler kinetic model, which can be evaluated at a fraction of the cost of the full model, with the same fluid limit behavior. We illustrate the method when applied to the solution of a kinetic equation of the type (\ref{eq:FP_general}) with deterministic interaction operator $Q(f,f)$ and random initial data $f({z},x,{\w},0)=f_0({z},x,{\w})$. 

\subsection{Variance reduction strategies}
Let us first consider the space homogeneous problem  
\be
\frac{\partial f}{\partial t} = Q(f,f),
\label{eq:BH}
\ee
where $f=f({z},\w,t)$ and with random initial data $f({z},{\w},0)=f_0({z},{\w})$. In \eqref{eq:BH}, without loss of generality, we have fixed $\varepsilon=1$ since in the space homogeneous case the Knudsen number scales with time. Under suitable assumptions, {one can show that} $f({z},{\w},t)$ exponentially decays to the unique steady state $f^{\infty}({z},{\w})$ such that $Q(f^\infty,f^\infty)=0$ which satisfies
\be
m_{\phi}(f_0)=m_{\phi}(f^\infty),\qquad m_{\phi}(f):=\int_{\RR^{d_{\w}}}\phi({\w})f({z},{\w},t)d{\w},
\label{eq:mom}
\ee
for some moments, for example $\phi({\w})=1,{\w},|{\w}|^2/2$ in the classical settings of conservation of mass, momentum and energy (see \cite{Vill, Trist, ToVil, Tos1999}). We emphasize that, for many space homogeneous kinetic models, the equilibrium state can be computed directly from the initial data thanks to the conservation properties \eqref{eq:mom}.

We recall that, given $M$ independent identically distributed (i.i.d.) samples $f^{k}(v,t)$, $k=1,\ldots,M$, of the solution to \eqref{eq:BH}, the standard Monte Carlo estimator reads
\be
E_M[f] = \frac1{M} \sum_{k=1}^M f^{k}(v,t).
\label{eq:MC}
\ee
The above estimator satisfies \cite{MSS,DPms}
\be
\|\EE[f] - E_M[f]\|_{{\LLBi}} \leq C {\sigma}_f M^{-1/2} ,
\ee
where ${\sigma}_f=\|\var(f)^{1/2}\|_{\LL}$ and $\var(f)=\EE[(\EE[f]-f)^2]$ is the variance.

A simple variance reduction strategy is obtained by splitting the expected value of the solution as
\be
\begin{split}
{\mathbb E}[f]({\w},t)&=\int_{\Omega} f({z},{\w},t)p({z})d{z}\\
&=\int_{\Omega} f^\infty({z},{\w})p({z})d{z}+
\int_{\Omega} (f({z},{\w},t)-f^\infty({z},{\w}))p({z})d{z}\\
&=\EE[f^\infty](v)+\EE[f-f^\infty](v,t),
\end{split}
\label{eq:mme}
\ee
and exploiting the fact that $\EE[f^\infty]$ can be evaluated with arbitrary accuracy at a negligible cost (for example using a very fine grid of samples) since it does not depend on the solution computed at each time step.

Now, if we use \eqref{eq:mme} and estimate 
\be
\EE[f]\approx \EE[f^\infty]+E_M[f-f^\infty]
\label{eq:28}
\ee 
we obtain an error of the type
\[ 
\| {\EE}[f](\cdot,t)-\EE[f^\infty](\cdot)-E_M[f-f^\infty](\cdot,t)\|_{{\LHBi}} \simeq \sigma_{f-f^\infty} M^{-1/2}.
\] 
Since the non equilibrium part $f-f^\infty$ goes to zero in time exponentially fast, then also its variance goes to zero, which means that for long times estimate \eqref{eq:28} becomes exact and depends only on the accuracy of the evaluation of $\EE[f^\infty]$.

The above argument can be generalized by considering a time dependent approximation of the solution $\tilde{f}({z},{\w},t)$, whose evaluation is significantly cheaper than computing $f({z},{\w},t)$, such that 
$m_\phi(\tilde{f})=m_\phi(f)$ for some moments and that
$\tilde{f}({z},{\w},t)\to f^{\infty}({z},{\w})$ as $t\to\infty$.  

For example, one can consider the space homogeneous Boltzmann equation \eqref{eq:BH} where $Q(f,f)$ is given by  \eqref{eq:Boltzmann} and assume 
 $\tilde{f}({z},{\w},t)$ to be the exact solution of the space homogeneous BGK approximation 
\be
\frac{\partial \tilde{f}}{\partial t} = \nu (\tilde{f}^\infty - \tilde{f}) 
\label{eq:BGKh}
\ee 
with $\nu>0$ independent from $z$ and for the same initial data $f_0({z},{\w})$. Thanks to the time invariance of the equilibrium state we have $\tilde{f}^\infty=f^\infty$ and we can write the exact solution to \eqref{eq:BGKh} as
\be
\tilde{f}({z},{\w},t) = e^{-\nu t} f_0({z},{\w})+(1-e^{-\nu t}) f^{\infty}({z},{\w}).
\label{eq:BGKexa}
\ee
We can assume that the expected value of the control variate $\EE[\tilde{f}]({\w},t)$ is computed with arbitrary accuracy at a negligible cost since it is a convex combination of the initial data and the equilibrium part. We denote this value by 
\be
\tilde{\bf f}(v,t)=e^{-\nu t} {\bf f}_0({\w})+(1-e^{-\nu t}) {\bf f}^{\infty}({\w}),
\label{eq:BGKexas}
\ee
where ${\bf f}_0={\mathbb E}[f_0(\cdot,{\w})]$ and ${\bf f}^{\infty}={\mathbb E}[f^{\infty}](\w)$ or accurate approximations of the same quantities.

If we now use the estimator 
\be
\EE[f]\approx \EE[\tilde f]+E_M[f-\tilde f]
\label{eq:29}
\ee 
we have an error like
\[ 
\| {\EE}[f](\cdot,t)-\EE[\tilde f](\cdot)-E_M[f-\tilde f](\cdot,t)\|_{{\LHBi}} \simeq \sigma_{f-\tilde f} M^{-1/2},
\] 
where even in this case, $\sigma_{f-\tilde f} \to 0$ as $t\to \infty$.

\subsubsection{Control variate estimators}\label{sec:single}

The simple variance reduction argument presented in the previous section can be used to formalize the following control variate estimator \cite{DPms}
\be
\tilde{E}^{\lambda}_M[f](v,t)=\frac1{M} \sum_{k=1}^M f^{k}(v,t) - \lambda\left(\frac1{M} \sum_{k=1}^M \tilde{f}^{k}(v,t)-\tilde{\bf f}({\w},t)\right).
\label{eq:nest2}
\ee
The control variate estimator \eqref{eq:nest2} is unbiased and consistent for any choice of $\lambda\in\RR$.
In particular, for $\lambda=0$ we recover the standard MC estimator $\tilde{E}^{0}_M[f]=E_M[f]$, whereas for $\lambda=1$ we have the  estimator $\tilde{E}^{1}_M[f]=\tilde{\bf f}+E_M[f-\tilde f]$ corresponding to \eqref{eq:29}.

If we now consider the random variable
\[
f^{\lambda}(z,v,t)=f(v,z,t)-\lambda(\tilde f(z,v,t)-\tilde{\bf f}(v,t)),
\]
we have ${\mathbb E}[f^\lambda]={\mathbb E}[f]$, $E_M[f^\lambda]=\tilde E_M^{\lambda}[f]$ and we can quantify its variance as 
\be
\var(f^\lambda)=\var(f)+\lambda^2 \var(\tilde f)-2\lambda\cov(f,\tilde f).
\label{eq:var1}
\ee
We have the following \cite{DPms}
\begin{theorem}
The quantity   
\be
\lambda^* = \frac{\cov(f,\tilde f)}{\var(\tilde f)}
\label{eq:lambdas}
\ee
minimizes the variance of $f^\lambda$ at the point $(v,t)$ and gives
\be
\var(f^{\lambda^*}) = (1-\rho_{f,\tilde f}^2)\var(f), 
\label{eq:var2}
\ee
where $\rho_{f,\tilde f} \in [-1,1]$ is the correlation coefficient between $f$ and $\tilde f$. In addition, we have  
\be
\lim_{t\to\infty} \lambda^*(v,t) =1,\qquad \lim_{t\to\infty} \var(f^{\lambda^*})(v,t)=0\qquad \forall\, v \in \RR^{d_v}.
\label{eq:lambdasa}
\ee
\label{th:1}
\end{theorem}
Using such an approach, in combination with a deterministic solver satisfying \eqref{eq:det}, one obtains the following error estimate \cite{DPms,MSS}
\bea
\|\EE[f](\cdot,t)-\tilde{E}^{\lambda^*}_M[f]\|_{{\LHBi}} \leq  C\left\{
\sigma_{f^{\lambda^*}}
M^{-1/2}+\Delta v^q\right\} 
\label{eq:errHMMC}
\eea
where $\sigma_{f^{\lambda^*}}=\|(1-\rho_{f,\tilde f}^2)^{1/2}\var(f)^{1/2}\|_{\LH}$, and $C>0$ depends on the final time and on the initial data. Here we ignored the statistical errors due to the approximation of the control variate expectation and to the estimate of $\lambda^*$. 
Note that, since $\rho^2_{f,\tilde{f}}\to 1$ as $t\to\infty$ the statistical error will  vanish for large times. 

\begin{remark}
In practice, $\cov(f,\tilde f)$ appearing in $\lambda^*$ is not known and has to be estimated. Starting from the $M$ samples we have the following unbiased estimators
\bea
\label{eq:varm}
{\var}_M(\tilde f) &=& \frac1{M-1}\sum_{k=1}^M (\tilde f^{k}-E_M[\tilde f])^2,\\
\label{eq:covm}
{\cov}_M(f,\tilde f) &=& \frac1{M-1}\sum_{k=1}^M (f^{k}-E_M[f])(\tilde f^{k}-E_M[\tilde f]),
\eea
which allow to estimate
\be
{\lambda}_M^*= \frac{{\cov}_M(f,\tilde f)}{{\var}_M(\tilde f)}.
\ee 
It can be verified easily that ${\lambda}_M^*\to 1$ as $f\to f^\infty$.  
\end{remark}

\subsection{Multiscale control variate estimators}
\label{sec:MSCVnh}
Let us now consider the full space non homogeneous problem \eqref{eq:FP_general}.
Again the idea is to compute the control variate function with a simplified model which can be evaluated at a fraction of the computational cost of the full model. However, in the space non homogeneous case the expectation of the control variate typically depends on time and must be estimated along the simulation.

Assuming the classical setting of the Boltzmann equation \eqref{eq:FP_general}-\eqref{eq:Boltzmann}, where the collision term characterizes conservation of mass, momentum and energy, integration of \eqref{eq:FP_general} with respect to the collision invariants $\phi({\w})=1,{\w},|{\w}|^2/2$ gives the moments equations
\begin{eqnarray}
\label{eq:smallscale}
\partial_t U +\partial_x
{\mathcal F}(U)+\nabla_x \int_{\RR^{d_{\w}}} {\w}\, \phi (f-f^\infty)\,d{\w}&=&0,
\end{eqnarray} 
where $U=(\rho,u,T)$, with density, mean velocity and temperature defined as 
\be
 \rho=\int_{\RR^{d_{\w}}} f d{\w}, \ u=\frac{1}{\rho}\int_{\RR^{d_{\w}}} {\w}\, f d{\w}, \ T=\frac{1}{d \rho}
 \int_{\RR^{d_{\w}}}|{\w}-u|^{2} f d{\w},
\label{eq:Mo} \ee
and
\[
{\mathcal F}(U)=\int_{\RR^{d_{\w}}} {\w}\, \phi f^\infty\,d{\w},\quad \phi({\w})=1,{\w},|{\w}|^2/2.
\]
Now, generalizing the space homogeneous method based on the local equilibrium $f_\infty$ as control variate we can consider the Euler closure as control variate, namely to assume $f=f^\infty$ in \eqref{eq:smallscale}. 
If we denote by $U_F=(\rho_F,u_F,T_F)^T$ the solution of the fluid model  
\be
\partial_t U_F +\partial_x {\mathcal F}(U_F) = 0,
\label{eq:Euler}
\ee
for the same initial data, the corresponding equilibrium state $f_F^\infty$ can be used as control variate.

Similarly to the homogeneous case, the generalization to an improved control variate based on a suitable approximation of the kinetic solution can be done with the aid of a more accurate fluid approximation, like the compressible Navier-Stokes system, or a simplified kinetic model. In the latter case, we can solve a BGK model
\be
\frac{\partial}{\partial t} \tilde{f} + {\w} \cdot \nabla_x \tilde{f} = \frac{\nu}{\varepsilon} (\tilde{f}^\infty-\tilde{f}),
\label{eq:BGK}
\ee 
for the same initial data and use its solution as control variate.

More precisely, given $M$ i.i.d. samples of the solution $f^k(x,v,t)$ and of the control variate $\tilde f^k(x,v,t)$ we define the estimator
\be
\tilde{E}^{\lambda}_M[f](x,v,t)=\frac1{M} \sum_{k=1}^M f^{k}(x,v,t) - \lambda\left(\frac1{M} \sum_{k=1}^M \tilde{f}^{k}(x,v,t)-\tilde{\bf f}(x,{\w},t)\right),
\label{eq:nesth2}
\ee 
where $\tilde{\bf f}(x,{\w},t)$ is an accurate approximation of ${\mathbb E}[\tilde{f}(\cdot,x,{\w},t)]$. 

The fundamental difference is that now the variance of
\[
f^{\lambda}(z,x,v,t)=f(z,x,v,t)-\lambda(\tilde{f}(z,x,{\w},t)-\tilde {\bf f}(z,x,{\w},t))
\]
will not vanish asymptotically in time since $f^\infty\neq \tilde{f}$, unless the kinetic equation is close to the fluid regime, namely for small values of the Knudsen number. Thus, the first part of Theorem \ref{th:1} is still valid, however the optimal value  
\be
\lambda^* = \frac{\cov(f,\tilde{f})}{\var(\tilde{f})}
\label{eq:lambdash2}
\ee
and the variance
\be
\var(f^{\lambda*})=(1-\rho^2_{f,\tilde f})\var(f)
\label{eq:varl2}
\ee
now satisfy   
\be
\lim_{\varepsilon\to 0} \lambda^*(x,v,t) =1,\qquad \lim_{\varepsilon\to 0} \var(f^{\lambda^*})(x,v,t)=0\qquad \forall\, (x,v) \in \RR^{d_x}\times\RR^{d_v}.
\label{eq:alambda}
\ee
In fact, since as $\varepsilon\to 0$ from \eqref{eq:FP_general}
 we formally have $Q(f,f)=0$ which implies $f=f^\infty$ and $\tilde{f} = f^\infty$, from \eqref{eq:lambdash2} and \eqref{eq:varl2} we obtain \eqref{eq:alambda}.

Even if simulating the control variate system is much cheaper than the full model, since its computational cost is no more negligible we cannot ignore it. Thus, we assume that the control variate model is computed over a fine grid of $M_E \gg M$ samples and use the approximation
\[
\tilde{\bf f}(x,{\w},t)=E_{M_E}[\tilde f](x,v,t),
\]  
in the estimator \eqref{eq:nesth2} which we will denote by $\tilde{E}^{\lambda}_{M,M_E}[f]$. 

This, however, has an impact on the optimal value of $\lambda$ in estimator \eqref{eq:nesth2}. In fact, 
using the independence of $E_M[\cdot]$ and $E_{M_E}[\cdot]$ and the fact that by the central limit theorem \cite{Lo77,HH} we have $\var(E_M[f])=M^{-1}\var(f)$, $\var(E_{M_E}[\tilde f])=M_E^{-1}\var(\tilde f)$, minimizing the variance now leads to the optimal value
\be
\tilde\lambda^* = \frac{M_E}{M+M_E}\lambda^*,
\ee
with $\lambda^*$ given by \eqref{eq:lambdash2}. This correction may be relevant in the cases when $M$ and $M_E$ do not differ too much.  
In our setting, however, $M_E \gg M$ so that $\frac{M_E}{M+M_E}\approx 1$ and we can assume $\tilde \lambda^* \approx \lambda^*$.

Using the optimal value \eqref{eq:lambdash2} and an underlaying deterministic solver which satisfies \eqref{eq:det}, we obtain the error estimate \cite{DPms, MSS}
\bea
\nonumber
&& \|E[f](\cdot,t)-\tilde{E}^{\lambda^*}_{M,M_E}[f]\|_{\LLBi}\\[-.2cm]
\label{eq:errHMMC2}
\\
\nonumber
&& \hskip 4cm \leq {C}\left\{\sigma_{f^{\lambda^*}} M^{-1/2}+\tau_{f^{\lambda^*}} M_E^{-1/2}+\Delta x^p+\Delta v^q\right\} 
\eea
where $\sigma_{f^{\lambda^*}}=\|(1-\rho_{f,\tilde f}^2)^{1/2}\var(f)^{1/2}\|_{\LL}$, $\tau_{f^{\lambda^*}}=\|\rho_{f,\tilde f} \var(f)^{1/2}\|_{\LL}$,  the constant $C>0$ depends on the final time and on the initial data. Now, $\rho^2_{f,\tilde{f}}\to 1$ as $\varepsilon\to 0$, therefore in the fluid limit we recover the statistical error of the fine scale control variate model.

\begin{remark}
For space non homogeneous simulations,  since we typically are interested in the evolution of the moments, one can compute the optimal value of $\lambda$ with respect to a given moment $m_{\phi}(f)$ as
\be
\lambda_{\phi}^* = \frac{\cov(m_{\phi}(f),m_{\phi}(f_F^\infty))}{\var(m_{\phi}(f_F^\infty))}
\label{eq:momcv}
\ee
where $m_{\phi}(\cdot)$ is defined by \eqref{eq:mom}, 
so that $\lambda^*_{\phi}=\lambda^*_{\phi}(x,t)$. 
This approach leads to a value of $\lambda^*_{\phi}$ which depends on the moment evaluated and, since $\lambda^*_{\phi}$ is independent of the velocity, strongly reduces the storage requirements. Note that, using \eqref{eq:momcv} all estimates in this section for $\EE[f]$ translates easily into estimates for $\EE[m_{\phi}(f_F^\infty)]$. 
\end{remark}

\section{Multiple multi-scale control variates}
The approach described in Section \ref{sec:MSCV} is fully general and accordingly to the particular kinetic model studied one can select a suitable approximated solution as control variate which acts at a given scale. In this section we extend the methodology to the use of several approximated solutions as control variates with the aim to further improve the variance reduction properties of MSCV methods.

\subsection{Multiple control variates}
\label{sec:mcv}
To keep notations simple we describe the method in the case of the space homogeneous equation \eqref{eq:BH}, the extension to the space non homogeneous case follows analogously to Section \ref{sec:MSCVnh} and is discussed at the end of the Section.

Let us consider ${f}_{1}(z,v,t),\ldots,{f}_{L}(z,v,t)$ approximations of $f(z,v,t)$ solution to \eqref{eq:BH} whose properties will be discussed later. 
We can define the random variable
\be
f^{\lambda_1,\ldots,\lambda_L}(z,v,t)=f(z,v,t)-\sum_{h=1}^L \lambda_h ({f}_{h}(z,v,t)-\EE[{f}_{h}](v,t)).
\label{eq:mcv}
\ee  
Clearly \eqref{eq:mcv} is such that $\EE[f^{\lambda_1,\ldots,\lambda_L}]=\EE[f]$ and has variance
\[
\var(f^{\lambda_1,\ldots,\lambda_L})=\var(f)+\sum_{h=1}^L \lambda_h^2 \var({f}_{h})+2\sum_{h=1}^L\lambda_h\left(\sum_{k=1 \atop k\neq h}^L \lambda_k
\cov({f}_{h},{f}_{k}) - \cov(f,{f}_{h})\right).
\] 
In vector form, if we introduce notations
\[
{\llambda}=(\lambda_1,\ldots,\lambda_L)^T,\quad
{b}=(\cov(f,{f}_{1}),\ldots,\cov(f,{f}_{L}))^T
\]
we have
\be
\var(f^{\llambda})=\var(f)+\llambda^T  C \llambda - 2\llambda^T b
\label{eq:varmcv}
\ee
where $ C=( c_{ij})$, ${ c_{ij}}=\cov({f}_{i},{f}_{j})$ is the symmetric $L\times L$ covariance matrix.

\begin{theorem}
Assuming the covariance matrix is not singular, the vector
\be
\llambda^* = {C}^{-1} b,
\label{eq:lambdamcv}
\ee
minimizes the variance of $f^\llambda$ at the point $(v,t)$ and gives
\be
\var(f^{\llambda^*})= \left(1-\frac{b^T(C^{-1})^T b}{\var(f)}\right)\var(f).
\label{eq:varmvcs}
\ee
\label{th:2}
\end{theorem}
\begin{proof}
To compute the minimizing values $\lambda^*_h$, $h=1,\ldots,L$ the first order optimality conditions are found by equating to zero the partial derivatives with respect to $\lambda_h$
\be
\frac{\partial\, \var(f^{\llambda})}{\partial \lambda_h} = 0, \quad h=1,\ldots,L. 
\label{eq:min}
\ee
This corresponds to solve the following linear system
\be
\cov(f,{f}_{h}) = \sum_{k=1}^L \lambda_k \cov({f}_{h},{f}_{k}),\quad h=1,\ldots, L,
\label{eq:sys1}
\ee
or in vector form
\be
b = {C} \llambda.
\ee
Therefore, assuming the covariance matrix is not singular, we obtain the solution \eqref{eq:lambdamcv}.
It is easily shown, via the second order optimality conditions that $\llambda^*$ is
indeed the variance-minimizing choice of $\llambda$. By direct substitution in \eqref{eq:varmcv} we obtain \eqref{eq:varmvcs}.
\end{proof}

The control variate estimator based on \eqref{eq:mcv} takes the form
\be
E^{\llambda}_M[f](v,t) = E_M[f](v,t)-\sum_{h=1}^L \lambda_h \left(E_{M}[f_h](v,t)-{\bf f}_{h}(v,t)\right),
\label{eq:mscvg}
\ee
where ${\bf f}_{h}(v,t)$ is an accurate approximation of $\EE[f_h](v,t)$, and we assumed to have $M$ i.i.d. samples from the solution $f(z,v,t)$ and the control variate functions $f_h(z,v,t)$ for $h=1,\ldots,L$. 
To estimate the value of the vector $\llambda^*$ we can use directly the Monte Carlo samples as in the scalar case. 
The resulting MSCV algorithm is summarized as follows:

\begin{algorithm}[Multiple MSCV - homogeneous case]~
\begin{enumerate}
\item {\bf Sampling}: 
Sample $M$ i.i.d. initial data from
the random initial data $f_0$ and approximate these over the grid $\Delta v$. Denote these samples by $f_{\Delta v}^{k,0}$, $k=1,\ldots,M$.
\item {\bf Solving}: 
\begin{enumerate}
\item 
For each control variate and for each realization of the random
input data $f_{\Delta v}^{k,0}$, $k=1,\ldots,M$, the resulting control variate model is solved numerically by a deterministic solver with mesh width $\Delta v$. We denote the resulting ensemble of deterministic solutions for $h=1,\ldots,L$ at time $t^n$ by
\[
f_{h,\Delta v}^{k,n},\quad k=1,\ldots,M.
\]
\item For each realization $f_{\Delta v}^{k,0}$, $k=1,\ldots,M$ 
the underlying kinetic equation (\ref{eq:BH}) is solved numerically by the corresponding deterministic solver with mesh width $\Delta v$. We denote the solution at time $t^n$ by $f^{k,n}_{\Delta {\w}}$, $k=1,\ldots,M$. 
\end{enumerate}
\item {\bf Estimating}: 
\begin{enumerate}
\item
Estimate the optimal vector of values $\llambda^*$ solving
\be
C^n_M \llambda^{*,n} = b^n_M,
\label{eq:ill}
\ee
where $(C^n_M)_{ij} = \cov_M(f_{i,\Delta v}^{n},f_{j,\Delta v}^{n})$ and $(b^n_M)_i=\cov_M(f_{\Delta v}^{n},f_{i,\Delta v}^{n})$.
\item
Compute the expectation of the random solution with the control variate estimator
\be
E^{\llambda^*}_M[f^n_{\Delta v}] = \frac1{M} \sum_{k=1}^M f^{k,n}_{\Delta {\w}}-\sum_{h=1}^L \lambda^{*,n}_{h} \left(\frac1{M} \sum_{k=1}^M {f}_{h,\Delta {\w}}^{k,n}-{\bf f}_{h,\Delta v}^n\right).
\label{mcest2}
\ee
\end{enumerate}
\end{enumerate}
\label{alg:1}
\end{algorithm}
Let us remark that, if we introduce the vector $F=(F_1,\ldots,F_L)^T$, such that $F_h=f_h-\EE[f_h]$,  $\EE[F_h]=0$, $h=1,\ldots,L$,
then equation \eqref{eq:mcv} reads 
\be
f^{\lambda_1,\ldots,\lambda_L}(z,v,t)=f(z,v,t)-\sum_{h=1}^L \lambda_h F_h(z,v,t),
\label{eq:mcv2}
\ee 
and thus, the variance of $f^{\llambda^*}$ is reduced to zero if $f$ is in the span of the set of functions  $F_1,\ldots,F_L$.

Using Gram--Schmidt orthogonalization, we may assume that the $L$ components of the control variate vector $F$ are orthogonal in the $L^2$ inner product
\[
\langle f,g \rangle = \int_{\Omega} f(z) g(z) p(z) dz.
\] 
In fact, we have
\[
\langle F_h,F_k \rangle=\cov(f_h,f_k)=\cov(F_h,F_k), \quad h,k=1,\ldots, L.
\]
We can construct the vector $G=(G_1,\ldots,G_L)^T$, with orthogonal components, $\langle G_h,G_k \rangle = 0$ for $h\neq k$, as follows \cite{GV}
\be
g_h = f_h - \sum_{j=1}^{h-1} \frac{\cov(g_j,f_h)}{\var(g_j)} g_j,\qquad h=1,\ldots,L,
\label{eq:gs}
\ee
and define $G_h=g_h-\EE[g_h]$, such that $\EE[G_h]=0$, $h=1,\ldots,L$.

Then we may try to minimize the variance of the random variable 
\be
f^{\gamma_1,\ldots,\gamma_L}(z,v,t)=f(z,v,t)-\sum_{h=1}^L \gamma_h G_h(z,v,t),
\label{eq:mcv3}
\ee
which now using the orthogonality property reads 
\[
\var(f^{\gamma_1,\ldots,\gamma_L})=\var(f)+\sum_{h=1}^L \gamma_h^2 \var(g_h)-2\sum_{h=1}^L \gamma_h\cov(f,g_h).
\] 
Denoting with $\Gamma=(\gamma_1,\ldots,\gamma_L)^T$, $D$ the diagonal matrix with elements $d_h=\var(g_h)$ and with $e$ the vector with components $e_h=\cov(f,g_h)$ we get
\be
\var(f^\Gamma)=\var(f)+\Gamma^T D \Gamma - 2\Gamma^T e.
\ee
By the same arguments as in Theorem \ref{th:2}, if the matrix $D$ is not singular, we have that the vector $\Gamma^*=D^{-1}e$ minimizes the above variance.
Thus we have proved the following result.
\begin{theorem}
If the control variate vector $G=(G_1,\ldots,G_L)^T$ in \eqref{eq:mcv3} has orthogonal components, $\langle G_h,G_k \rangle = 0$ for $h\neq k$, then if $\langle G_h,G_h \rangle \neq 0$ the vector $\Gamma^*$ with components
\be
\gamma^*_h = \frac{\cov(f,g_h)}{\var(g_h)},\qquad h=1,\ldots,L,
\label{eq:gammah}
\ee
minimizes the variance of $f^\Gamma$ at the point $(v,t)$ and gives
\be
\var(f^{\Gamma_*})=\left(1-\sum_{h=1}^L \rho^2_{f,g_h}\right)\var(f)
\label{eq:var3}
\ee
where $\rho_{f,g_h}\in [-1,1]$ is the correlation coefficient between $f$ and $g_h$. 
\end{theorem}
Estimating the orthogonal set of control variates using $M$ samples by
\be
E^{\Gamma}_M[f](v,t) = E_M[f](v,t)-\sum_{h=1}^L \gamma_h \left(E_{M}[g_h](v,t)-{\bf g}_h(v,t)\right),
\label{eq:mscvg2}
\ee
where ${\bf g}_h(v,t)=\EE[g_h](v,t)$ or its accurate approximation, 
in combination with a deterministic solver satisfying \eqref{eq:det}, one obtains the following 
result \cite{DPms,MSS}.
\begin{proposition}
Consider a deterministic scheme which satisfies \eqref{eq:det} in the velocity space for the solution of the homogeneous kinetic equation \eqref{eq:BH} with deterministic interaction operator $Q(f,f)$ and random initial data $f({z},{\w},0)=f_0({z},{\w})$. Assume that the initial data is sufficiently regular.

Then, the MSCV estimate defined in \eqref{eq:mscvg2} with the optimal values given by \eqref{eq:gammah} satisfies the error bound 

\bea
\|\EE[f](\cdot,t^n)-{E}^{\Gamma^*}_M[f^n_{\Delta v}]\|_{{\LHBi}} \leq  C\left\{\sigma_{f^{\Gamma^*}} M^{-1/2}+\Delta v^q\right\} 
\label{eq:errHMMC2b}
\eea
where $\sigma_{f^{\Gamma^*}}=\left\|\left(1-\sum_{h=1}^L \rho^2_{f,g_h}\right)^{1/2}\var(f)^{1/2}\right\|_{\LH}$,
and $C>0$ depends on the final time and on the initial data. 
\end{proposition}
\begin{proof}
The bound follows from
\begin{eqnarray*}
\nonumber
\|\EE[f](\cdot,t^n)-{E}^{\Gamma^*}_M[f^{n}_{\Delta {\w}}]\|_{{\LHBi}}&& \\
\nonumber
&& \hskip -2cm \leq \|\EE[f](\cdot,t^n)-{E}^{\Gamma^*}_M[f](\cdot,t^n)\|_{{\LLBi}}\\
\\[-.25cm]
\nonumber
&& \hskip -2cm + 
\|{E}^{\Gamma^*}_M[f](\cdot,t^n)-{E}^{\Gamma^*}_M[f_{\Delta \w}^n]\|_{{\LLBi}}\\
\nonumber
&& \hskip -2cm \leq  C\left\{\sigma_{f^{\Gamma^*}}  M^{-1/2}+\Delta v^q\right\} , 
\end{eqnarray*}
where the Monte Carlo bound in the first term now make use of \eqref{eq:var3} and the second term is bounded by the discretization error of the deterministic scheme.
\end{proof}

Here we ignored the statistical errors due to the approximation of the control variates expectations and to the estimate of the vector $\Gamma^*$.

\subsubsection{Two control variates}\label{sec:twoc}
To exemplify the approach, it is interesting to consider the case $L=2$, where $f_1 = f_0$, the initial data, and $f_2 = f^\infty$, the stationary state. In this case we know that $f$ is in the span of the control variates at $t=0$ and as $t\to\infty$.

A straightforward computation shows that the optimal values $\lambda^*_1$ and $\lambda_2^*$ are given by
\begin{eqnarray}
\nonumber
\lambda_1^* &=& \frac{\var(f^\infty)\cov(f,f_0)-\cov(f_0,f^\infty)\cov(f,f^\infty)}{\Delta},\\[-.2cm]
\label{eq:l12}
\\[-.2cm]
\nonumber
\lambda_2^* &=& \frac{\var(f_0)\cov(f,f^\infty)-\cov(f_0,f^\infty)\cov(f,f_0)}{\Delta},
\end{eqnarray}
where $\Delta = \var(f_0)\var(f^\infty)-\cov(f_0,f^\infty)^2\neq 0$.

Using $M$ samples for both control variates the optimal estimator reads
\be
E^{\lambda^*_1,\lambda^*_2}_M[f](v,t) =E_M[f](v,t)-\lambda^*_1 \left(E_M[f_0](v)-{\bf f}_0(v)\right)-\lambda^*_2 \left(E_M[f^\infty](v)-{\bf f}^\infty(v)\right).
\label{eq:nest2b}
\ee

Now, at $t=0$ since $f(z,v,0)=f_0(z,v)$ we clearly have $\lambda_1^*=1$ and $\lambda_2^*=0$ so that the estimator \eqref{eq:nest2b} is exact \[E^{1,0}_M[f](v,0)={\bf f}_0(v).\] Moreover, 
by the same arguments as in Theorem \ref{th:1}, for large times since $f(z,v,t)\to f^\infty(z,v)$ from \eqref{eq:l12} we get 
\[
\lim_{t\to\infty} \lambda_1^* = 0, \qquad \lim_{t\to\infty} \lambda_2^* = 1,
\]
and thus, the variance of the estimator vanishes asymptotically in time
\[
\lim_{t\to\infty} E^{\lambda^*_1,\lambda^*_2}_M[f](v,t) = E^{0,1}_M[f](v)={\bf f}^\infty(v).
\]
We emphasize that this last example can be seen as a generalization of the scalar case based on the BGK model. In fact, the estimator  \eqref{eq:nest2} based on \eqref{eq:BGKexa}-\eqref{eq:BGKexas}  can be written in the form \eqref{eq:nest2b} as
\[
E^{\lambda^*}_M[f](v,t) =E_M[f](v,t)-\tilde\lambda^*_1 \left(E_M[f_0](v)-{\bf f}_0(v)\right)-\tilde\lambda^*_2 \left(E_M[f^\infty](v)-{\bf f}^\infty(v)\right)
\]
where
\[
\tilde\lambda_1^* = e^{-t} \lambda^*,\qquad \tilde\lambda_2^* = (1-e^{-t}) \lambda^*.
\]
Therefore, the scalar control variate based on the BGK model can be understood as a suboptimal solution to the our minimization problem for the control variates $f_0$ and $f^\infty$. In particular, if the solution $f$ has the form \eqref{eq:BGKexa}, namely the full model is the BGK model, then it is in the span generated by $f_0$ and $f^\infty$ and we obtain $\lambda_1^*=\tilde\lambda_1^*$ and  $\lambda_2^*=\tilde\lambda_2^*$.

\subsection{Hierarchical multiple control variates}
The multiple control variate approach just described presents some limitations. First the linear system \eqref{eq:ill} may be difficult to solve due to ill conditioning of the covariance matrix and additionally the  estimation of the coefficients in the matrix requires the use of a control variate independent number of samples. This last aspect can represent a serious limitation in space non homogeneous situations in which the control variates may be originated from models that operate at the various spatio-temporal scales of the problem with different levels of complexity.   

To overcome such drawbacks here we formulate a recursive construction of the multiple control variate estimator \eqref{eq:mscvg}. To this aim, let us assume that the control variates $f_1,\ldots,f_L$ represent kinetic models with an increasing level of fidelity. Under this assumption the control variate $f_1$ represents the less accurate model whereas the control variate $f_L$ is the closer model to the full model $f$.

To start with, we estimate $\EE[f]$ with $M_L$ samples using $f_L$ as control variate
\[
\EE[f]\approx E_{M_L}[f]-{\hat \lambda}_L\left(E_{M_L}[f_L]-\EE[f_L]\right).
\]
Next, to estimate of $\EE[f_L]$ we use $M_{L-1}\gg M_L$ samples and consider $f_{L-1}$ as control variate 
\[
\EE[f_L]\approx E_{M_{L-1}}[f_L]-{\hat \lambda}_{L-1}\left(E_{M_{L-1}}[f_{L-1}]-\EE[f_{L-1}]\right).
\]
Similarly, in a recursive way we can construct estimators for the remaining expectations of the control variates $\EE[f_{L-2}],\EE[f_{L-3}],\ldots,\EE[f_2]$ using respectively $M_{L-3}\ll M_{L-4}\ll\ldots \ll M_1$ samples
until 
\[
\EE[f_2]\approx E_{M_{1}}[f_2]-{\hat \lambda}_{1}\left(E_{M_{1}}[f_{1}]-\EE[f_{1}]\right),
\]
and we stop with the final estimate
\[
\EE[f_1]\approx E_{M_0}[f_1],
\]
with $M_0 \gg M_1$.

By combining the estimators of each stage together we obtain the recursive MSCV estimator
\begin{eqnarray}
\nonumber
E_L^{{r,\hat \llambda}}[f] &=& E_{M_L}[f] -{\hat \lambda}_L\left(E_{M_L}[f_L]-E_{M_{L-1}}[f_L]\right.\\
\nonumber
&+&{\hat \lambda}_{L-1}\left(E_{M_{L-1}}[f_{L-1}]-E_{M_{L-2}}[f_{L-1}]\right.\\[-.3cm]
\label{eq:lambdar}
\\[-.2cm]
\nonumber
&\ldots&\\
\nonumber
&+&{\hat \lambda}_{1}\left.\left.\left(E_{M_{1}}[f_{1}]-E_{M_{0}}[f_{1}]\right) \ldots\right)\right).
\end{eqnarray}
Now if we compute the optimal values ${\hat \lambda}_h^*$ independently for each recursive stage by ignoring the errors due to the approximations of the various expectations, if $\var(f_{h})\neq 0$, we obtain
\be
{\hat \lambda}_h^* = \frac{\cov(f_{h+1},f_{h})}{\var(f_{h})},\quad h=1,\ldots,L
\label{eq:lambdasr}
\ee
where we used the notation $f_{L+1}=f$. We refer to this set of values, which avoids the solution of the resulting linear system, as quasi-optimal.
Note that, since the control variates $f_{h+1}$ and $f_h$ are known on the same set of samples $M_h$ the values ${\hat \lambda}^*_h$ can be estimated using \eqref{eq:varm}-\eqref{eq:covm}.  

The estimator \eqref{eq:lambdar} can be recast in the form 
\bea
\nonumber
E_L^{{\hat \llambda}}[f] &=&  E_{M_{L}}[f_{L+1}]-\sum_{h=1}^{L}\lambda_h(E_{M_h}[f_h]-E_{M_{h-1}}[f_h])\\[-.2cm]
\label{eq:mscvgr}
\\[-.2cm]
\nonumber
&=&\lambda_1 E_{M_0}[f_1] + \sum_{h=1}^{L} (\lambda_{h+1} E_{M_h}[f_{h+1}]-\lambda_{h} E_{M_{h}}[f_{h}]),
\eea
where we defined 
\be
\lambda_h=\prod_{j=h}^L \hat\lambda_j,\quad h=1,\ldots,L,\qquad \lambda_{L+1}=1.
\label{eq:lambdah}
\ee
Since by the central limit theorem \cite{Lo77,HH} we have $\var(E_M[f])=M^{-1}\var(f)$, using the independence of the estimators $E_{M_h}[\cdot]$, $h=0,\ldots,L$, the total variance of the estimator \eqref{eq:mscvgr} is 
\bea
\nonumber
\var(E_L^{{\hat \llambda}}[f])&=&\lambda_1^2 M_0^{-1} \var(f_1) \\[-.2cm]
\\[-.2cm]
\nonumber
&+& \sum_{h=1}^{L} M_h^{-1}\left\{\lambda_{h+1}^2\var(f_{h+1})+\lambda_{h}^2\var(f_{h})-2\lambda_{h+1}\lambda_{h}\cov(f_{h+1},f_h)\right\}.
\eea
Now, the first order optimality conditions 
\[
\frac{\partial\,\var(E_L^{{\hat \llambda}}[f])}{\partial \lambda_h} = 0,\qquad h=1,\ldots,L
\]
leads to the tridiagonal system for $h=1,\ldots,L$
\[
M_{h-1}^{-1}\left\{\lambda_h\var(f_h)-\lambda_{h-1} \cov(f_h,f_{h-1})\right\}+
M_{h}^{-1}\left\{\lambda_h \var(f_h)-\lambda_{h+1}\cov(f_{h+1},f_{h})\right\}=0,
\]
or equivalently 
\be
\lambda_h \var(f_h) - \lambda_{h-1} \frac{M_{h}}{M_{h-1}+M_{h}} \cov(f_{h},f_{h-1}) -
\lambda_{h+1}\frac{M_{h-1}}{M_{h-1}+M_{h}} \cov(f_{h+1},f_{h}) =0
\label{eq:sys2}
\ee
where we assumed $\lambda_{0}=0$ and $\lambda_{L+1}=1$. The resulting system can be solved efficiently by the usual Thomas algorithm for tridiagonal systems provided $\var(f_h)\neq 0$, $h=1,\ldots,L$.

System \eqref{eq:sys2} can be rewritten as
\[
\lambda_h \var(f_h)  -
\lambda_{h+1}\cov(f_{h+1},f_{h}) =\frac{M_{h}}{M_{h-1}+M_{h}}\left(\lambda_{h-1}  \cov(f_{h},f_{h-1}) + \lambda_{h+1} \cov(f_{h+1},f_{h})\right)
\]
which, reverting to the original control variate variables becomes 
\[
\hat\lambda_h \var(f_h)  -
\cov(f_{h+1},f_{h}) =\frac{M_{h}}{M_{h-1}+M_{h}}\left(\hat\lambda_{h}\hat\lambda_{h-1}  \cov(f_{h},f_{h-1}) +  \cov(f_{h+1},f_{h})\right).
\]
Thus, the quasi-optimal values \eqref{eq:lambdasr} solves the above system  up to $O({M_{h}}/({M_{h-1}+M_{h}}))$.

\begin{theorem}
The vector $\Lambda^*=(\lambda_1^*,\ldots,\lambda_L^*)^T$ solution of the tridiagonal system \eqref{eq:sys2} minimizes the variance of the estimator \eqref{eq:mscvgr}. In particular the vector $\hat \Lambda^*=(\hat\lambda_1^*,\ldots,\hat\lambda_L^*)^T$ of quasi-optimal solutions given by \eqref{eq:lambdasr} is such that 
\be
\prod_{j=h}^L \hat\lambda^*_j=\lambda^*_h + O\left(\bar{\mu}_h\right),\quad h=1,\ldots,L
\label{th:ll}
\ee
where $\bar{\mu}_h = \displaystyle\max_{h \leq k \leq L}\left\{\displaystyle\frac{M_{k}}{M_{k-1}+M_{k}}\right\}$.
\end{theorem}
\begin{proof}
We can rewrite \eqref{eq:sys2} in the form $C \Lambda = b$ where $C=\hat C + {\cal M} C_0$ with
\[
{\hat C} = \left(
\begin{array}{cccc}
 c_{11} &  - c_{12}  &  &0\\
 0 & \ddots   &\ddots   &\\
     & \ddots   & \ddots  & -c_{L-1 L}\\
0  &   &0 & c_{LL}  
\end{array}
\right),\qquad C_0=\left(
\begin{array}{cccc}
 0 &  c_{12}  &  &0\\
 -c_{21} & \ddots   &\ddots   &\\
     & \ddots   & \ddots  & c_{L-1 L}\\
0  &   &-c_{L L-1} & 0  
\end{array}
\right),
\]
$c_{ij}=\cov(f_{i},f_{j})$, ${\cal M}={\rm diag}\{\mu_1,\ldots,\mu_L\}$, $\mu_h={M_{h}}/({M_{h-1}+M_{h}})$ and $b = (I-{\cal M}){\hat b}$, ${\hat b}=\left(0,\ldots,0,\cov(f,f_L)\right)^T$. By construction the vector $\Lambda^*$, such that $(\hat C + {\cal M} C_0)\Lambda^* = (I-{\cal M}){\hat b}$, minimizes the variance of \eqref{eq:mscvgr}.

Let us define the vector $\tilde\Lambda^*$ of elements $\prod_{j=h}^L \hat\lambda^*_j$, $h=1,\ldots,L$ where $\hat\lambda^*_j$ are given by \eqref{eq:lambdasr}. We have $\hat C \tilde \Lambda^* = \hat b$ 
so that 
\[
\hat{C}(\Lambda^*-{\tilde \Lambda^*})=-{\cal M}(\hat b + C_0 \Lambda^*)
\]
therefore if $\var(f_h)\neq 0$, $h=1,\ldots,L$ we can write
\[
{\tilde \Lambda^*}=\Lambda^*+{\hat C}^{-1}{\cal M}(\hat b + C_0 \Lambda^*).
\]
Now, since ${\hat C}^{-1}$ is upper triangular and ${\cal M}$ diagonal we have  
 \eqref{th:ll}. 
\end{proof}

We can summarize the details of the method, when applied to the space homogeneous problem \eqref{eq:BH} in combination with a deterministic solver, in the following algorithm.

\begin{algorithm}[Recursive multiple MSCV - homogeneous case]~
\begin{enumerate}
\item {\bf Sampling}: 
For each control variate $f_h$, we draw a number
$M_h$ of i.i.d. samples from the random initial data $f_0$ and approximate these over the mesh $\Delta {\w}$. 
Denote these control variate dependent number of samples for $h=1,\ldots,L$ by
$$f_{h,\Delta v}^{k,0},\quad k=1,\ldots,M_h$$ and set $f_{\Delta v}^{k,0}=f_{L,\Delta v}^{k,0},\quad k=1,\ldots,M_L$.
\item {\bf Solving}: 
\begin{enumerate}
\item 
For each control variate and for each realization of the random
input data $f_{h,\Delta v}^{k,0}$, $k=1,\ldots,M_h$, the resulting control variate model is solved numerically by a deterministic solver with mesh widths $\Delta v$. We denote the resulting ensemble of deterministic solutions for $h=1,\ldots,L$ at time $t^n$ by
\[
f_{h,\Delta v}^{k,n},\quad k=1,\ldots,M_h.
\]
\item For each realization $f_{\Delta v}^{k,0}$, $k=1,\ldots,M_L$ 
the underlying kinetic equation (\ref{eq:FP_general}) is solved numerically by the corresponding deterministic solver with mesh widths $\Delta v$. We denote the solution at time $t^n$ by $f^{k,n}_{\Delta {\w}}$, $k=1,\ldots,M_L$. 
\end{enumerate}
\item {\bf Estimating}: 
\begin{enumerate}
\item
Estimate the optimal vector of values $\hat\llambda^*$ as
\be
\hat\lambda^{*,n}_{h}=\frac{\cov_{M_h}(f_{h+1,\Delta v}^{n},f_{h,\Delta v}^{n})}{\var_{M_h}(f_{h,\Delta v}^{n})},\quad h=1,\ldots,L
\label{eq:hlambdam}
\ee
where we used the notation $f_{L+1,\Delta v}^{k,n}=f_{\Delta v}^{k,n}$, $k=1,\ldots,M_L$.
\item
Compute the expectation of the random solution with the control variate estimator
\be
E^{\hat\llambda^*}_L[f^n_{\Delta v}] = \frac1{M_L} \sum_{k=1}^{M_L} f^{k,n}_{\Delta {\w}}-\sum_{h=1}^L \lambda^{*,n}_{h} \left(\frac1{M_h} \sum_{k=1}^{M_h} {f}_{h,\Delta {\w}}^{k,n}-{\bf f}_{h,\Delta v}^n\right),
\label{mcestr2}
\ee
where
\[
{\bf f}_{h,\Delta v}^n=\frac1{M_{h-1}} \sum_{k=1}^{M_{h-1}} {f}_{h,\Delta {\w}}^{k,n},\qquad \lambda^{*,n}_h=\prod_{j=h}^L \hat\lambda^{*,n}_j,\quad h=1,\ldots,L. 
\]
\end{enumerate}
\end{enumerate}
\label{alg:2}
\end{algorithm}

Regarding the error bound that we obtain using \eqref{eq:mscvgr} with the values given by \eqref{eq:lambdasr} le us observe that if, at each stage, we denote
\[
{E}^{\hat\lambda_h}_{M_h}[f_h] = E_{M_h}[f_h]-{\hat\lambda_h}\left(E_{M_h}[f_{h-1}]-E_{M_{h-1}}[f^n_{h-1}]\right)
\]
then by the error bound \eqref{eq:errHMMC2} we have
\[
\|\EE[f_h](\cdot,t)-{E}^{\hat\lambda_h^*}_{M_h}[f_{h}](\cdot,t)\|_{{\LHBi}} \leq  C_h\left\{\sigma_h
 M_h^{-1/2}+\tau_h M^{-1/2}_{h-1}\right\} 
\]
where $C_h>0$ is a suitable constant and we defined
\be
\sigma_h=\left\|\left(1-\rho^2_{f_h,f_{h-1}}\right)^{1/2}\var(f_h)^{1/2}\right\|_{\LH},\qquad \tau_h=\|\rho_{f_h,f_{h-1}}\var(f_h)^{1/2}\|_{\LH}.
\ee
Using the recursive estimator, in combination with a deterministic solver satisfying \eqref{eq:det}, we can write
\bea
\nonumber
&& \|\EE[f](\cdot,t^n)-{E}^{\hat\Lambda^*}_{L}[f^n_{\Delta v}]\|_{{\LHBi}} \\
\nonumber
&& \hskip 4cm \leq  
\|\EE[f](\cdot,t^n)-{E}^{\hat\Lambda^*}_{L}[f](\cdot,t^n)\|_{{\LHBi}}\\
\nonumber
&& \hskip 4cm +
\|{E}^{\hat\Lambda^*}_{L}[f](\cdot,t^n)-{E}^{\hat\Lambda^*}_{L}[f^n_{\Delta v}]\|_{{\LHBi}}.
\eea 
The second term is bounded as usual by the discretization error of the scheme, whereas, ignoring the statistical errors in estimating the quasi-optimal vector of values $\hat\Lambda^*$, the first term can be estimated recursively as
\bea
\nonumber
&&\|\EE[f](\cdot,t^n)-{E}^{\hat\Lambda^*}_{L}[f](\cdot,t^n)\|_{{\LHBi}}\\
\nonumber
&& \hskip 4cm \leq C_L 
\left\{\sigma_L M_L^{-1/2}\right.\\
\nonumber
&& \hskip 4cm +
|\|\hat\lambda_L^*(\EE[f_{L-1}](\cdot,t^n)-{E}^{\hat\Lambda^*}_{L-1}[f_{L-1}](\cdot,t^n))\|_{{\LHBi}}\bigg\}
\\
\label{eq:errREC}
&& \hskip 4cm \leq C_{L}
\left\{\sigma_L M_L^{-1/2}+\tau_L C_{L-1}\left\{\sigma_{L-1} M_{L-1}^{-1/2}\right.\right.\\
\nonumber
&& \hskip 4cm +  \|\hat\lambda_{L-1}^*(\EE[f_{L-2}](\cdot,t^n)-{E}^{\hat\Lambda^*}_{L-2}[f_{L-2}](\cdot,t^n))\|_{{\LHBi}}\bigg\}
\\
\nonumber
&&\hskip 4cm \ldots\\
\nonumber
&&\hskip 4cm \leq C \left(\sum_{h=1}^L \xi_h \sigma_h M_h^{-1/2}+\xi_0 M_0^{-1/2}\right)
\eea
where we defined
\be
\xi_h = \prod_{j=h+1}^L \tau_j.
\label{eq:muh}
\ee
Thus we have proved the following result 
\begin{proposition}
Consider a deterministic scheme which satisfies \eqref{eq:det} in the velocity space for the solution of the homogeneous kinetic equation \eqref{eq:BH} with deterministic interaction operator $Q(f,f)$ and random initial data $f({z},{\w},0)=f_0({z},{\w})$. Assume that the initial data is sufficiently regular.

Then, the recursive MSCV estimate defined in \eqref{eq:lambdar} with satisfies the error bound 
\bea
\nonumber
&&\|\EE[f](\cdot,t^n)-{E}^{\hat\Lambda_h^*}_{L}[f^n_{\Delta v}]\|_{{\LHBi}} \\[-.2cm]
\label{eq:errREC2}
\\[-.2cm]
\nonumber
&& \hskip 4cm \leq C \left(\sum_{h=1}^L \xi_h \sigma_h M_h^{-1/2}+\xi_0 M_0^{-1/2}+\Delta v^q \right)
\eea
where $\xi_h$ are given by \eqref{eq:muh}, and
$C>0$ depends on the final time and on the initial data. 
\end{proposition}

Finally, concerning the relations between the recursive MSCV estimator \eqref{eq:mscvgr} and the Multi-level Monte Carlo (MLMC) approach \cite{MSS, DPms}, we made the following remarks. 

\begin{remark}~
\begin{itemize}
\item
We can emphasize the analogies with MLMC by using as control variates a hierarchy of discretizations of the kinetic equation with random inputs. In the case of a cartesian grid this aims at constructing a velocity discretization with corresponding mesh width $\Delta {\w}_h$ that satisfy $\Delta {\w}_h = 2^{1-h}(\Delta {\w}_1)$, $h=1,\ldots, L$ where $\Delta {\w}_1$ is the mesh width for the coarsest resolution. To avoid unessential difficulties, we denote by $f_h(z,v,t)$, $h=1,\ldots, L$, the continuous representation, for example by polynomial interpolation, of the corresponding numerical solution at time $t$ obtained with the underlaying deterministic method using mesh width $\Delta \w_h$.

Under these assumptions, fixing $\lambda_h=1$, $h=1,\ldots,L$ in \eqref{eq:mscvgr}, we get the classical MLMC estimator \cite{Giles}
\be
E_L^{{\bf 1}}[f](v,t) = E_{M_0}[f_1] + \sum_{h=1}^{L} (E_{M_h}[f_{h+1}-f_{h}]),
\label{eq:MLMC}
\ee
where we used the notation ${\bf 1}=(1,\ldots,1)^T$. Note that, as a side result of our derivation, using the quasi-optimal values \eqref{eq:lambdasr} in \eqref{eq:mscvgr} (or the optimal values solution to \eqref{eq:sys2}) with the hierarchical grid constructed above we obtain a quasi-optimal (optimal) version of MLMC.

\item Similarly to MLMC methods, in the recursive MSCV estimator \eqref{eq:mscvgr} the largest number of samples $M_0$ is required on the less accurate model $f_1$, where the samples are cheaper, whereas only a small number $M_{L}$ of samples are needed on the full model. The two key differences between the recursive MSCV estimator \eqref{eq:mscvgr} and the MLMC approach \eqref{eq:MLMC} consist in the choice of low fidelity models as control variates instead of a hierarchy of discretizations and in using the quasi-optimal values \eqref{eq:lambdasr} (or the optimal values solution to \eqref{eq:sys2}) instead of fixing $\lambda_h=1$, $h=1,\ldots,L$.

\end{itemize}
\end{remark}

\subsection{Multiple multi-scale control variates estimators}
For non homogeneous problems, as in the scalar case summarized in Section \ref{sec:MSCVnh}, the main difference is that we cannot assume to know the expectation of the control variate or that it can be computed accurately at a negligible computational cost. Each control variate, in fact, acts at a certain scale and requires the numerical solution of a suitable time dependent model in the phase space.

Now, the MSCV estimator \eqref{eq:mscvg}, based on the multiple control variates $f_1(z,x,v,t)$, $\ldots$, $f_L(z,x,v,t)$ for the solution $f(z,x,v,t)$ to problem \eqref{eq:FP_general}, reads 
\be
E^{\Lambda}_{M,M_E}[f](x,v,t) = E_M[f](x,v,t)-\sum_{h=1}^L \lambda_h \left(E_{M}[f_h](x,v,t)-E_{M_{E}}[f_h](x,v,t)\right),
\label{eq:mscvgh}
\ee
where $M_E \gg M$ samples have been used to estimate the expectations of the control variates $\EE[{f}_{h}(z,v,t)]$.

As in the scalar case, minimization of the variance of \eqref{eq:mscvgh},
leads to the optimal values
\be
\tilde\Lambda^* = \frac{M_E}{M+M_E}\Lambda,
\label{eq:lambdamcv2}
\ee
with $\Lambda^*$ given by \eqref{eq:lambdamcv}. In the sequel we assume 
 $M_E \gg M$ so that $\frac{M_E}{M+M_E}\approx 1$.

The extension of algorithms \ref{alg:1} to the non homogeneous case is reported below. 

\begin{algorithm}[Multiple MSCV - non homogeneous case]~
\begin{enumerate}
\item {\bf Sampling}:
\begin{enumerate}
\item Sample $M_E$ i.i.d. initial data $f_0^k$, $k=1,\ldots,M_E$ from the random initial data $f_0$ and approximate these over the grid characterized by $\Delta x$ and $\Delta v$.
\item Sample $M \ll M_E$ i.i.d. initial data $f_0^k$, $k=1,\ldots,M$ from the random initial data $f_0$ and approximate these over the grid characterized by $\Delta x$ and $\Delta v$.
\end{enumerate}

\item {\bf Solving}: 
\begin{enumerate}
\item 
For each control variate and for each realization  $f_{\Delta v}^{k,0}$, $k=1,\ldots,M_E$, the resulting control variate model is solved numerically by a deterministic solver with mesh widths $\Delta x,\Delta v$. We denote the resulting deterministic solutions for $h=1,\ldots,L$ at time $t^n$ by $f_{h,\Delta x, \Delta v}^{k,n}$, $k=1,\ldots,M_E$ and estimate their expectations by
\[
{\bf f}_{h,\Delta x,\Delta {\w}}^{n}=\frac1{M_E}\sum_{k=1}^{M_E}{f}_{h,\Delta x,\Delta v}^{k,n}.
\] 
\item For each realization $f_{\Delta x, \Delta v}^{k,0}$, $k=1,\ldots,M$ 
the underlying kinetic equation (\ref{eq:FP_general}) is solved numerically by the corresponding deterministic solver with mesh widths $\Delta x,\Delta v$. We denote the solution at time $t^n$ by $f^{k,n}_{\Delta x,\Delta {\w}}$, $k=1,\ldots,M$. 
\end{enumerate}
\item {\bf Estimating}: 
\begin{enumerate}
\item
Estimate the optimal vector of values $\llambda^*$ solving
\be
C^n_{M_E} \llambda^n = b^n_M,
\label{eq:ill2}
\ee
where $(C^n_{M_E})_{ij} = \cov_{M_E}(f_{i,\Delta x, \Delta v}^{n},f_{j,\Delta x, \Delta v}^{n})$ and $(b^n_M)_i=\cov_M(f_{\Delta x, \Delta v}^{n},f_{i,\Delta x, \Delta v}^{n})$.
\item
Compute the expectation of the random solution with the control variate estimator
\be
E^{\llambda^*}_{M,M_E}[f^n_{\Delta x, \Delta v}] = \frac1{M} \sum_{k=1}^M f^{k,n}_{\Delta x,\Delta {\w}}-\sum_{h=1}^L \lambda^{*,n}_{h} \left(\frac1{M} \sum_{k=1}^M {f}_{h,\Delta x,\Delta {\w}}^{k,n}-{\bf f}_{h,\Delta x, \Delta v}^n\right).
\label{mcest2ih}
\ee
\end{enumerate}
\end{enumerate}
\label{alg:1b}
\end{algorithm}
In the space non homogeneous case, similarly to the case where the control variates expectations are known, we can apply the Gram--Schmidt orthogonalization procedure \eqref{eq:gs} and use the estimator 
\be
E^{\Gamma}_{M,M_E}[f](v,t) = E_M[f](x,v,t)-\sum_{h=1}^L \gamma_h \left(E_{M}[g_h](x,v,t)-E_{M_E}[g_h](x,v,t)\right),
\label{eq:mscvg2ih}
\ee
with the optimal vector of values $\Gamma^*$ defined by \eqref{eq:gammah}.
For the estimator \eqref{eq:mscvg2ih} we have the following generalization of the error estimate \eqref{eq:errHMMC2b}.
\begin{proposition}
Consider a deterministic scheme which satisfies \eqref{eq:det} for the solution of the kinetic equation of the form \eqref{eq:FP_general} with deterministic interaction operator $Q(f,f)$ and random initial data $f(z,x,{\w},0)=f_0(z,x,{\w})$. Assume that the initial data is sufficiently regular.

Then, the MSCV estimate defined in \eqref{eq:mscvg2ih} with the optimal values given by \eqref{eq:gammah} satisfies the error bound 

\bea
\nonumber
&&\|\EE[f](\cdot,t^n)-{E}^{\Gamma^*}_{M,M_E}[f^n_{\Delta x,\Delta v}]\|_{\LLBi} \\[-.2cm]
\label{eq:errHMMC2bb}
\\
\nonumber
&& \hskip 2cm
\leq  C\left\{\sigma_{f^{\Gamma^*}} M^{-1/2}+\tau_{f^{\Gamma^*}} M_{E}^{-1/2}+\Delta {x}^p+\Delta v^q\right\} 
\eea
with $\sigma^2_{f^{\Gamma^*}}=\left\|\left(1-\sum_{h=1}^L \rho^2_{f,g_h}\right)\var(f)\right\|_{\LL}$, $\tau^2_{f^{\Gamma^*}}=\left\|\sum_{h=1}^L \rho^2_{f,g_h}\var(f)\right\|_{\LL}$,
and $C>0$ depends on the final time and on the initial data. 
\end{proposition}
\begin{proof}
We have
\begin{eqnarray*}
\nonumber
\|\EE[f](\cdot,t^n)-{E}^{\Gamma^*}_{M,M_E}[f^{n}_{\Delta x,\Delta {\w}}]\|_{\LLBi} &&\\
&& \hskip -2.5cm\leq\|\EE[f](\cdot,t^n)-{E}^{\Gamma^*}_{M}[f^{n}_{\Delta x,\Delta {\w}}]\|_{\LLBi} \\
\nonumber
&&\hskip -2.5cm+\|{E}^{\Gamma^*}_{M}[f^{n}_{\Delta x,\Delta {\w}}]-{E}^{\Gamma^*}_{M,M_E}[f^{n}_{\Delta x,\Delta {\w}}]\|_{\LLBi}\\
&&\hskip -2.5cm = I_1+I_2. 
\end{eqnarray*}
The first term $I_1$ can be bounded similarly to \eqref{eq:errHMMC2b} to get 
\begin{eqnarray*}
\nonumber
&&\|\EE[f](\cdot,t^n)-{E}^{\Gamma^*}_{M}[f^{n}_{\Delta x,\Delta {\w}}]\|_{\LLBi} \\
&&\hskip 4cm \leq  C_1\left\{\sigma_{f^{\Gamma^*}} M^{-1/2}+\Delta {x}^p+\Delta v^q\right\}. 
\end{eqnarray*}

Using the fact that, ignoring the statistical error in estimating $\gamma_h^{*}$, from \eqref{eq:mscvg2} and \eqref{eq:mscvg2ih} we have
\[
{E}^{\Gamma^*}_{M}[f^{n}_{\Delta x,\Delta {\w}}]-{E}^{\Gamma^*}_{M,M_E}[f^{n}_{\Delta x,\Delta {\w}}] = \sum_{h=1}^L \gamma_h^{*,n} \left(\EE[{g}^n_{h,\Delta x,\Delta v}]-E_{M_E}[g^n_{h,\Delta x,\Delta v}]\right).
\]
From \eqref{eq:gammah} the second term $I_2$ can be bounded by
\begin{eqnarray*}
\nonumber
&&\|{E}^{\Gamma^*}_{M}[f^{n}_{\Delta x,\Delta {\w}}]-{E}^{\Gamma^*}_{M,M_E}[f^{n}_{\Delta x,\Delta {\w}}]\|_{\LLBi}\\
&& \hskip 4cm \leq C_2\left\{\tau_{f^{\Gamma^*}} M_E^{-1/2}+\Delta {x}^p+\Delta {\w}^q\right\}. 
\end{eqnarray*}
\end{proof}

Similarly, the recursive MSCV estimator \eqref{eq:mscvgr}, based on a hierarchy of multiple control variates $f_1(z,x,v,t)$, $\ldots$, $f_L(z,x,v,t)$ with increasing level of fidelity for the solution $f(z,x,v,t)$ to problem \eqref{eq:FP_general}, is
\be
E^{\Lambda}_L[f](x,v,t) = E_{M_L}[f](x,v,t)-\sum_{h=1}^L \lambda_h \left(E_{M_h}[f_h](x,v,t)-E_{M_{h-1}}[f_h](x,v,t)\right),
\label{eq:mscvgrh}
\ee
with $M_{h-1} \gg M_h$ and where now the optimal values of $\Lambda^* = (\lambda_1^*,\ldots,\lambda_L^*)^T$ are obtained from the quasi-optimal solution \eqref{eq:lambdasr} using \eqref{eq:lambdah} or by the correction introduced by the solution of the tridiagonal system \eqref{eq:sys2} if relevant.

In this case, the extension of algorithms \ref{alg:2} and estimate \eqref{eq:errREC2} to the non homogeneous case follows straightforwardly simply replacing $f^{n,k}_{\Delta v}$ and $f^{n,k}_{h,\Delta v}$ with $f^{n,k}_{\Delta x,\Delta v}$ and $f^{n,k}_{h,\Delta x,\Delta v}$, and is omitted for brevity.

Finally, due to its importance in practical applications, we describe the details of the hierarchical method in the case $L=2$.

\subsubsection{Two multi-scale hierarchical control variates}\label{sec:twoh}

Let us focus on the case of a recursive multiscale estimator with two control variates $L=2$. To this aim we consider $f_1(z,x,v,t)$ as the equilibrium state $f_F^\infty(z,x,v,t)$ associated to the system of Euler equations 
\be
\partial_t U_F +\partial_x {\mathcal F}(U_F) = 0, 
\label{eq:Euler2}
\ee
with $U_F=(\rho_F,u_F,T_F)^T$, and corresponding to the limit case $\varepsilon\to 0$ in \eqref{eq:FP_general}. As a second control variate we consider $f_2(z,x,v,t)$ as  the solution of the BGK model 
\be
\frac{\partial}{\partial t}{f_2} + {\w} \cdot \nabla_x {f_2} = \frac{\nu}{\varepsilon} (f_2^\infty-{f_2}).
\label{eq:BGK2}
\ee
Both models are solved for the same initial data $f_0(z,x,v)$. Now, the Euler equations are used as control variate to improve the computation of the expectation in the BGK model, that in turn is used as control variate to improve the computation of the expectation in the full Boltzmann model. 

The recursive estimator now reads
\begin{eqnarray}
E_2^{{\hat\lambda_1,\hat\lambda_2}}[f] = E_{M_2}[f] -{\hat \lambda}_2\left(E_{M_2}[f_2]-E_{M_{1}}[f_2]
+{\hat \lambda}_{1}\left(E_{M_{1}}[f_{1}]-E_{M_{0}}[f_{1}]\right)\right),
\end{eqnarray}
where $M_0 \gg M_1 \gg M_2$. If we define $\lambda_2=\hat\lambda_2$ and $\lambda_1=\hat\lambda_1\hat\lambda_2$  their optimal values are computed as solutions of system \eqref{eq:sys2} for $L=2$
\bea
\nonumber
\lambda_1 \var(f_1)  -
\lambda_{2}(1-\mu_1) \cov(f_{2},f_{1}) &=&0\\
\nonumber
\lambda_2 \var(f_2) - \lambda_{1} \mu_2 \cov(f_{2},f_{1}) &=&
(1-\mu_2) \cov(f,f_{2}).
\eea  
with $\mu_h = M_h/(M_{h-1}+M_h)$, which gives
\bea
\lambda_1^* &=& \frac{(1-\mu_1)(1-\mu_2)\cov(f_{2},f_{1}) \cov(f,f_{2})}{\displaystyle\var(f_1)\var(f_2)-(1-\mu_1)\mu_2 \cov(f_{2},f_{1})^2}\\
\lambda_2^* &=& \frac{ (1-\mu_2) \var(f_{1}) \cov(f,f_{2})}{\var(f_1)\var(f_2)-(1-\mu_1)\mu_2\cov(f_{2},f_{1})^2}.  
\eea
The quasi-optimal values are obtained assuming $\mu_1,\mu_2\approx 0$ and are characterized by 
\be
\hat\lambda_1^* = \frac{\cov(f_2,f_1)}{\var(f_1)},\qquad \hat\lambda_2^* = \frac{\cov(f,f_2)}{\var(f_2)}. 
\ee
In the fluid limit $\varepsilon\to 0$ we have $f,f_1,f_2\to f^\infty$ so that
\[
\lim_{\varepsilon\to 0} \lambda_1^* = \frac{(1-\mu_1)(1-\mu_2)}{1-(1-\mu_1)\mu_2}=\frac{{M_{0}}}{M_0+M_{1}+M_{2}},\quad \lim_{\varepsilon\to 0} \lambda_2^* = \frac{(1-\mu_2)}{1-(1-\mu_1)\mu_2}= \frac{M_0+M_1}{M_0+M_{1}+M_{2}},
\]
then
\[
\lim_{\varepsilon\to 0} E_2^{{\lambda_1^*,\lambda_2^*}}[f] = \frac{M_0 E_{M_0}[f^\infty]+M_1 E_{M_1}[f^\infty]+M_2E_{M_2}[f^\infty]}{M_0+M_{1}+M_{2}}.
\]
On the contrary, the quasi-optimal values are such that
\[
\lim_{\varepsilon\to 0} \hat\lambda_1^* = 1,\qquad \lim_{\varepsilon\to 0} \hat\lambda_2^* = 1
\]
and therefore
\[
\lim_{\varepsilon\to 0} E_2^{{\hat\lambda_1^*,\hat\lambda_2^*}}[f] = E_{M_0}[f^\infty]
\]
which corresponds to the equilibrium solution over the finest grid of samples.

\section{Numerical examples}
\label{sec:4}
In this Section, we discuss several numerical examples with the aim of illustrating
the characteristics of the multiple control variate strategies described in the previous Sections. We focus on
the two control variates approach \ref{sec:twoc} and on the multi-scale hierarchical multiple control variates one in the case where the hierarchy consists of two models with an increasing level of fidelity as described in Section \ref{sec:twoh}. We use shorthands MC, MSCV, MSCV2 and MSCVH2 to denote, respectively, the standard Monte Carlo, the single multi-scale control variate, the multiple multi-scale control variates and the hierarchical multiple multi-scale control variates.
  
The first test problem consists of a space homogeneous case with uncertain initial data, the second test problem is a Riemann problem with uncertainties in the initial state while the third test problem has randomness in the boundary condition. These tests are analogous to those considered in \cite{DPms} using a single control variate and our precise aim is to show that with a multiple control variates approach it is possible to achieve better results in terms of the ratio between computational cost and accuracy. 
In all our numerical tests the velocity space is two dimensional $d_v=2$, the velocity domain is truncated to $[-v_{\min},v_{\max}]^2$ and the collision integral is solved by the fast spectral method \cite{MP,DP15}. In the sequel, all the statistical errors have to be intended in the sense of the average where $M_a=10$ values have been used to compute the mean.

\subsection{Space homogeneous Boltzmann equation} 
We consider the space homogeneous Boltzmann equation (\ref{eq:BH}). We compare the case of the single control variate based on the BGK model \ref{sec:single} with the case of the two control variates approach of Section \ref{sec:twoc}. The number of samples used to compute the expected solution for the Boltzmann equation is $M=100$. Note that, since the control variates do not contribute to the overall computational cost their expected values are computed with high accuracy using orthogonal polynomials.
\subsubsection{Test 1. Uncertain initial data}
The initial condition is a two bumps problem with uncertainty
\be
f_0(z,v)=\frac{\rho_0}{2\pi} \left(\exp\left(-\frac{|v-(2+sz)|^2}{\sigma}\right)+\exp\left({-\frac{|v+(1+sz)|^2}{\sigma}}\right)\right)
\ee  
with $s=0.2$, $\rho_0=0.125$, $\sigma=0.5$ and $z$ uniform in $[0,1]$. The velocity space is discretized with $N_v=64^2$ points.  We choose $v_{\min}=v_{\max}=16$. The time integration has been performed with a $4$-th order Runge-Kutta method using $\Delta t=0.05$ and a final time $T_f=10$.

In Figure \ref{Figure1}, we report the $L_2$ error with respect to the random variable in the computation of the expected value for the distribution function $\EE[f](v,t)$ for the various methods.  Clearly, all MSCV methods provide a gain in accuracy of several orders of magnitudes with respect to standard Monte Carlo. In particular, we observe that MSCV2 method based on two control variates permits to gain one order of accuracy with respect to the standard MSCV approach.

\begin{figure}
	\begin{center}
		\includegraphics[width=.6\textwidth]{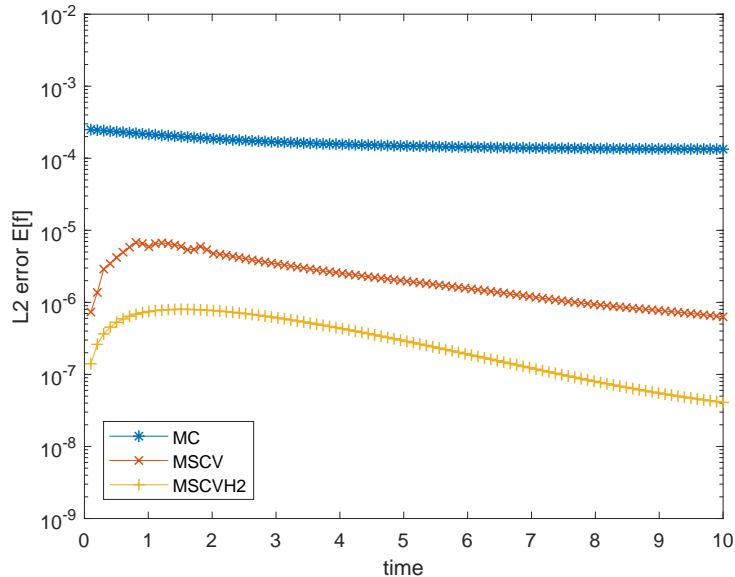}
		\caption{Test 1. $L_2$ norm of the error in time for the expectation of the distribution function for the MC method, the MSCV method based on the BGK solution and the MSCV2 method based on the two control variates $f_0$ and $f^\infty$.}
		\label{Figure1}
	\end{center}
\end{figure}

In Figure \ref{Figure4} we report the shape of the optimization coefficients $\lambda_1^*(v,t)$ and $\lambda_2^*(v,t)$ at the final time. It is possible to observe that $\lambda_1^*(v,t)$ is approaching zero and $\lambda_2^*(v,t)$ is approaching one, except in the regions where the solution still differs from its equilibrium value.

\begin{figure}
	\begin{center}
		\includegraphics[width=.45\textwidth]{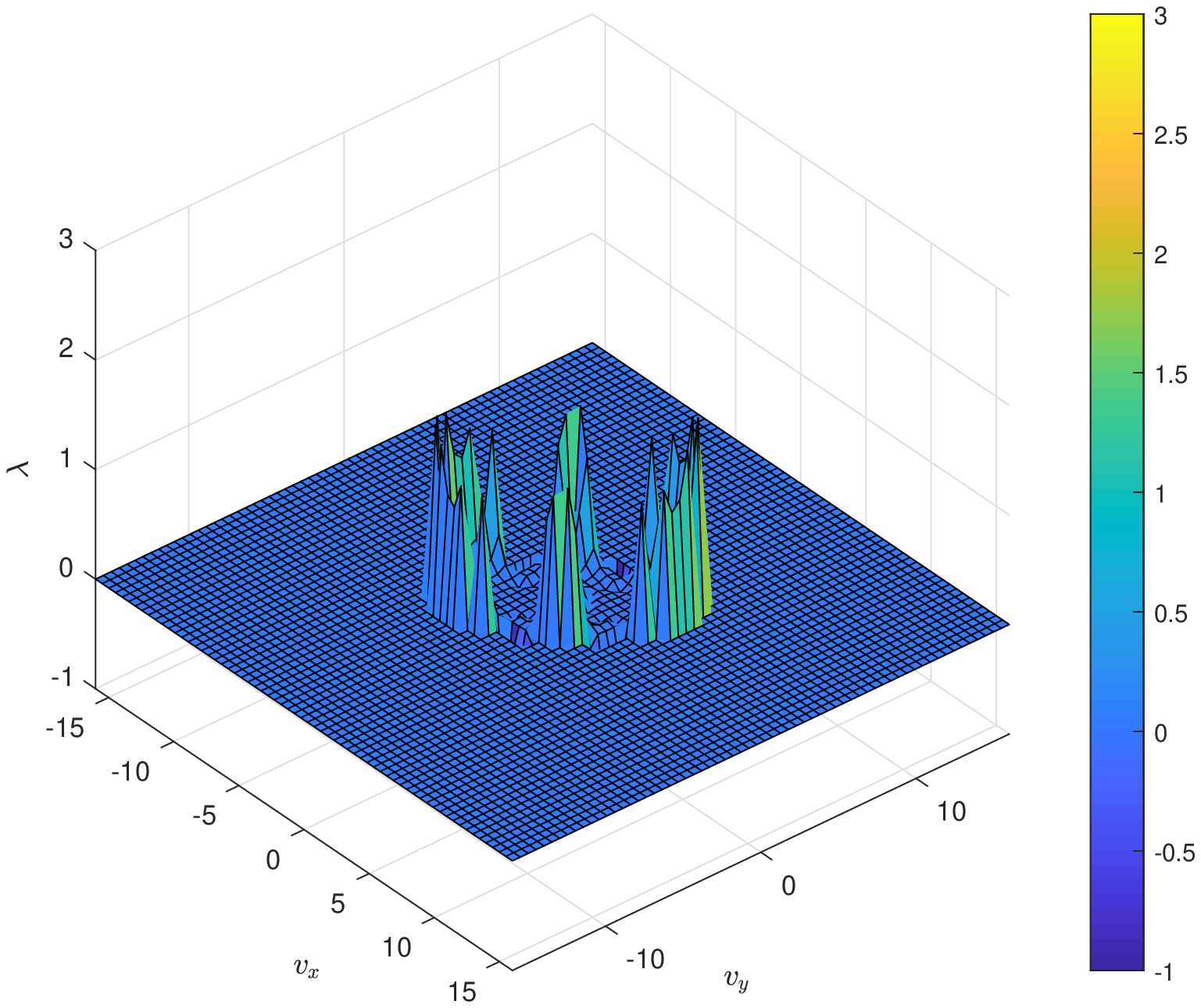}
		\includegraphics[width=.45\textwidth]{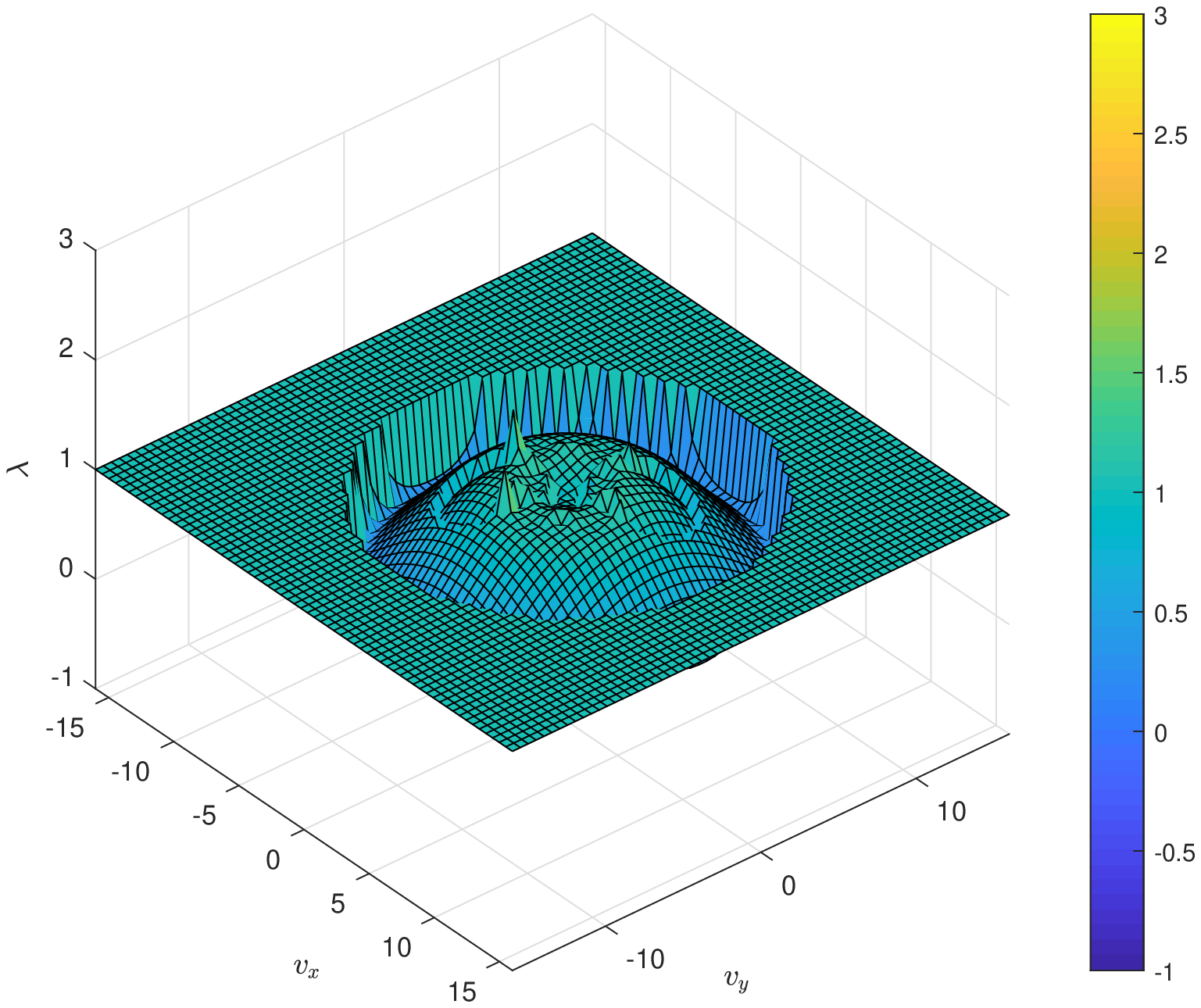}	
		\caption{Test 1. Optimal coefficients $\lambda_1^*(v,t)$ (left) and $\lambda_2^*(v,t)$ (right) at the final time $T_f=10$.}
		\label{Figure4}
	\end{center}
\end{figure}

\subsection{Space non homogeneous Boltzmann equation}
Next, we consider the space non homogeneous Boltzmann equation \eqref{eq:FP_general}. Concerning the space-time discretization, we make use of a fifth order WENO method \cite{JS} for the spatial discretization and a second order explicit Runge-Kutta method for the time discretization. The time step is the same for all methods and is taken as $\Delta t=\min\{\Delta x/(2 v_{\max}), \varepsilon\}$ with $\varepsilon$ the Knudsen number. Let observe that other time integrators can be used to solve the problem related to the stiffness of the collision operator which avoid the strict CFL condition coming from the solution of the relaxation phase (see \cite{DP15} for example). However, since this is not the core of the problem treated here, we will not discuss this issue in the rest of the paper.  

In all test cases we focus on the use of the BGK model and the Euler equations as control variates.
Since the ratio between accuracy and computational cost is of primary importance, we report in the following some estimates about the computational cost of the different models employed in the simulations.

The cost involved in the solution of the Boltzmann equation can be estimated by 
\[ 
\C(f)=C N_a^{d_v-1} N_v^{d_v}\log_2(N_v^{d_v})N_x^{d_x}
\]
the cost of the solution of the BGK equation by
\[ 
\C(\tilde f)= C_1 N_v^{d_v}N_x^{d_x}
\]
while the cost of the solution of the compressible Euler system by 
\[ 
\C(f_F^\infty) = C_2 N_x^{d_x}, 
\]
where $C, \ C_1$ and $C_2$ are suitable constants, $N_a$ is the number of angular directions in the velocity space used in the spectral discretization of the Boltzmann operator, $N_v$ the number of grid points in velocity space, $N_x$ the number of grid points in physical space and $d_v$ and $d_x$ respectively the dimensions in velocity and space. We use  $N_x=100$, $N_v=32$, $d_x=1$, $N_a=8$ and $d_v=2$ while a rough estimation of the ratio between the coefficients $C$, $C_1$ and $C_2$ gives ${C}/{C_1} \approx 1.25$ and ${C}/{C_2} \approx 1$. Concerning the number of samples of the random variable, we employ $M=10$ points for the Boltzmann model while the numbers $M_{E_1}\gg M$ and $M_{E_2}\gg M$ of samples used for the control variates models are discussed in the sequel.

The multiple control variates strategy is applied here in two different ways:
\begin{itemize}
\item The hierarchical method described in Section \ref{sec:twoh} based on the BGK model and the Euler system.
\item The two control variates method of Section \ref{sec:mcv} where two different BGK models are used. The first one uses $\nu_1=\rho$ while the second one $\nu_2=0.125\rho$ as relaxation frequencies. This choice is due to the fact that it is known that the BGK model tends to over-relax to the equilibrium state compared to the standard Boltzmann operator.
\end{itemize}
The standard MSCV method is computed using the BGK model as control variate.

\begin{figure}
	\begin{center}
		\includegraphics[width=.43\textwidth]{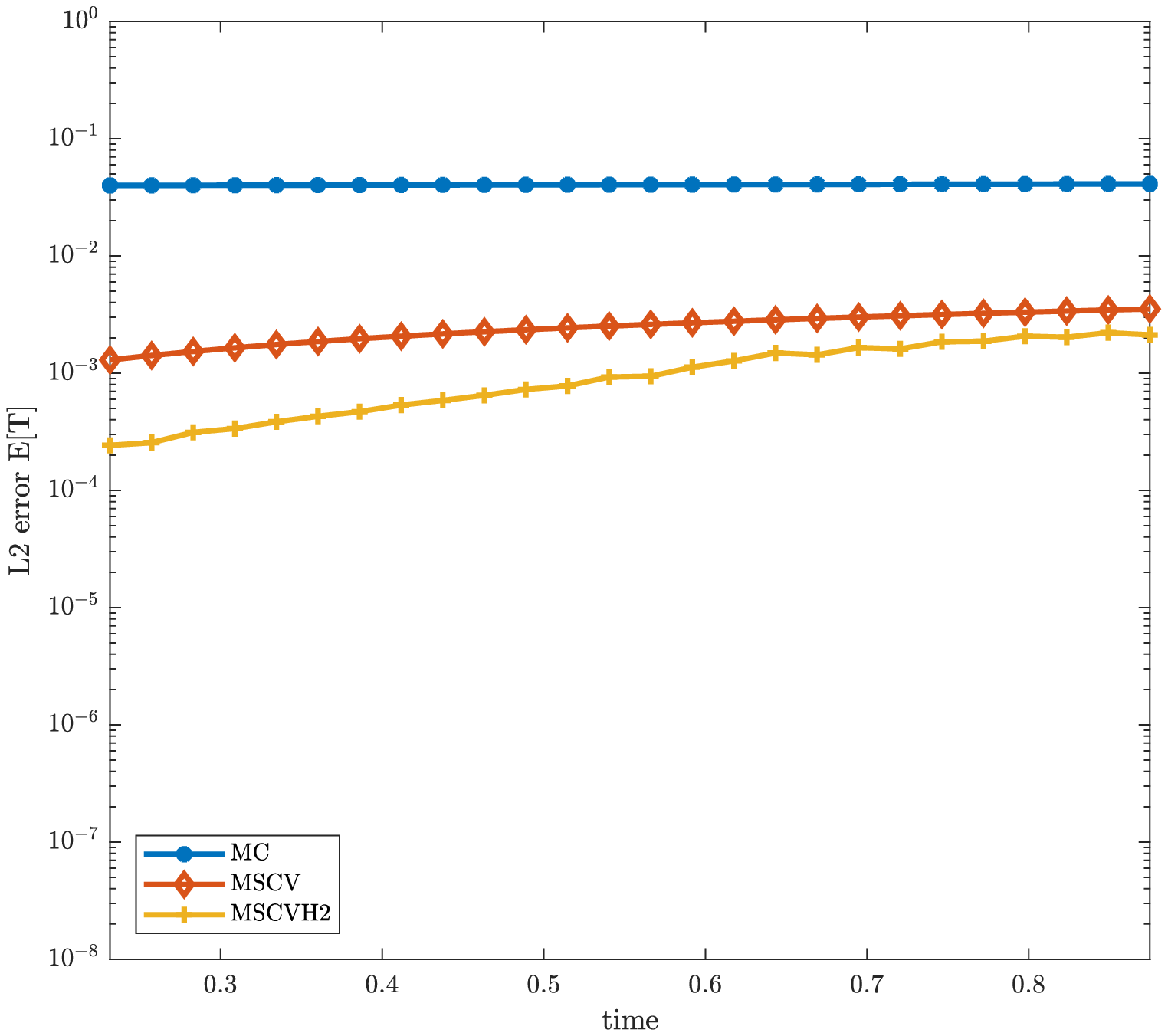}
		\includegraphics[width=.43\textwidth]{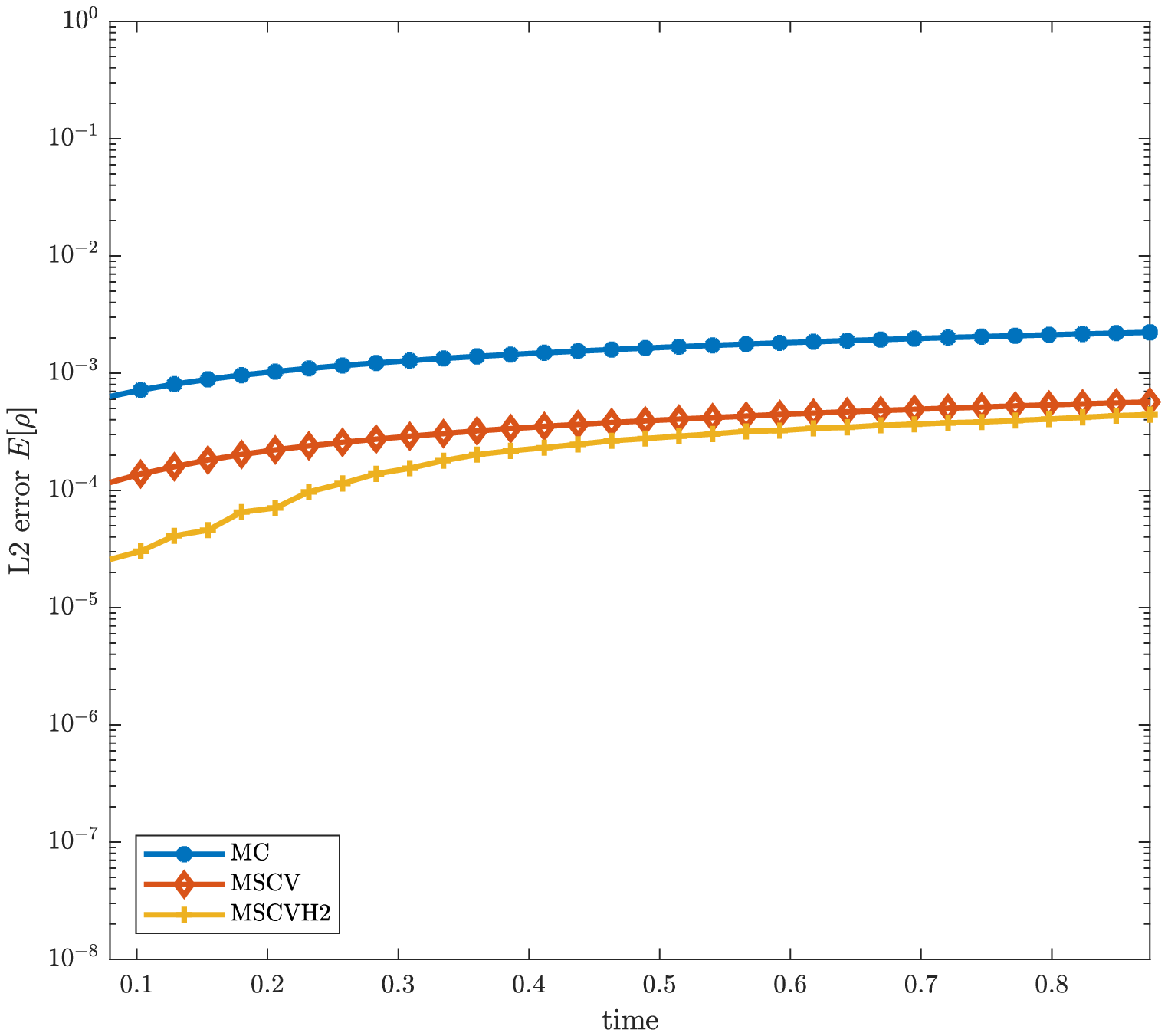}\\
		\includegraphics[width=.43\textwidth]{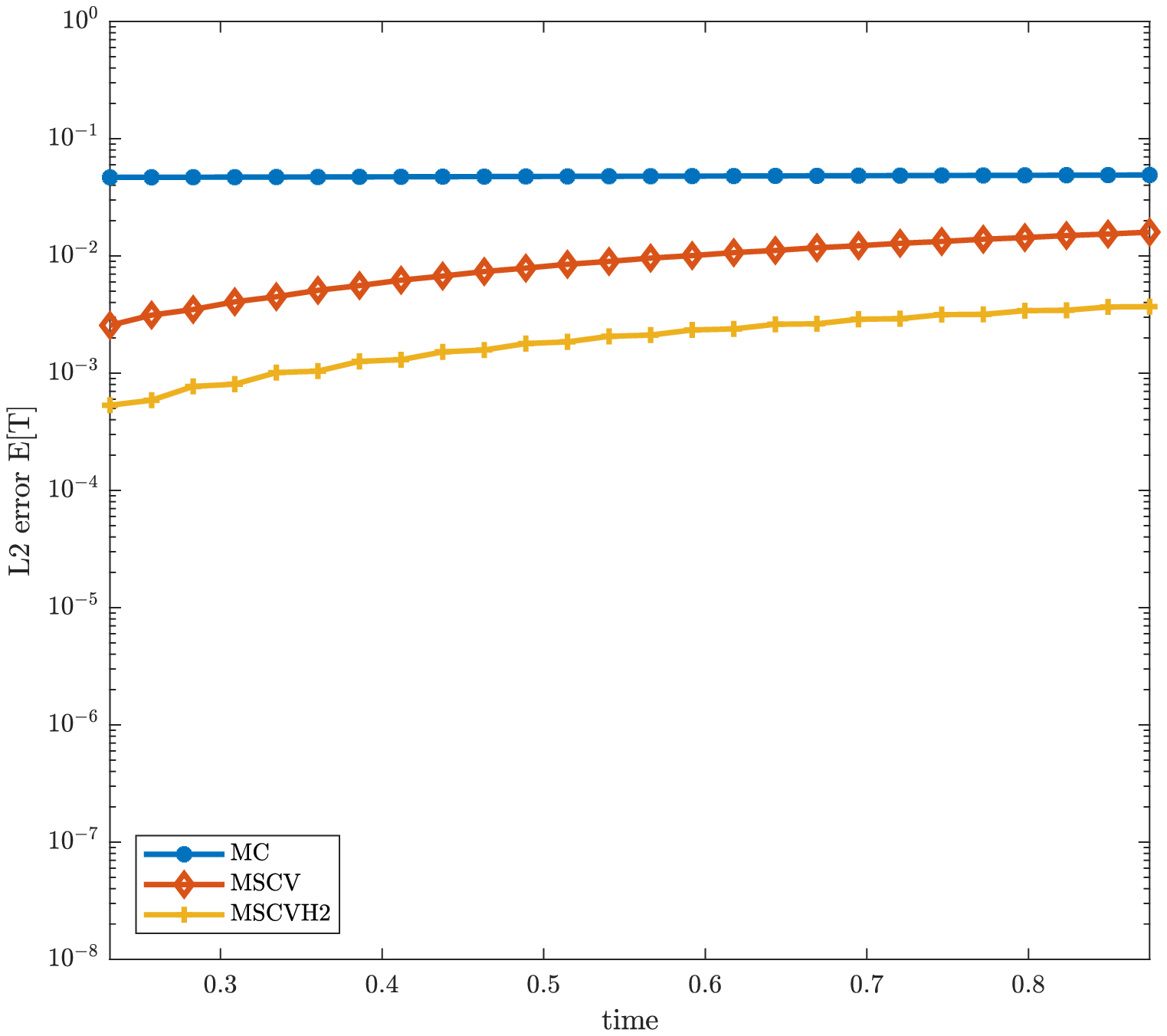}
		\includegraphics[width=.43\textwidth]{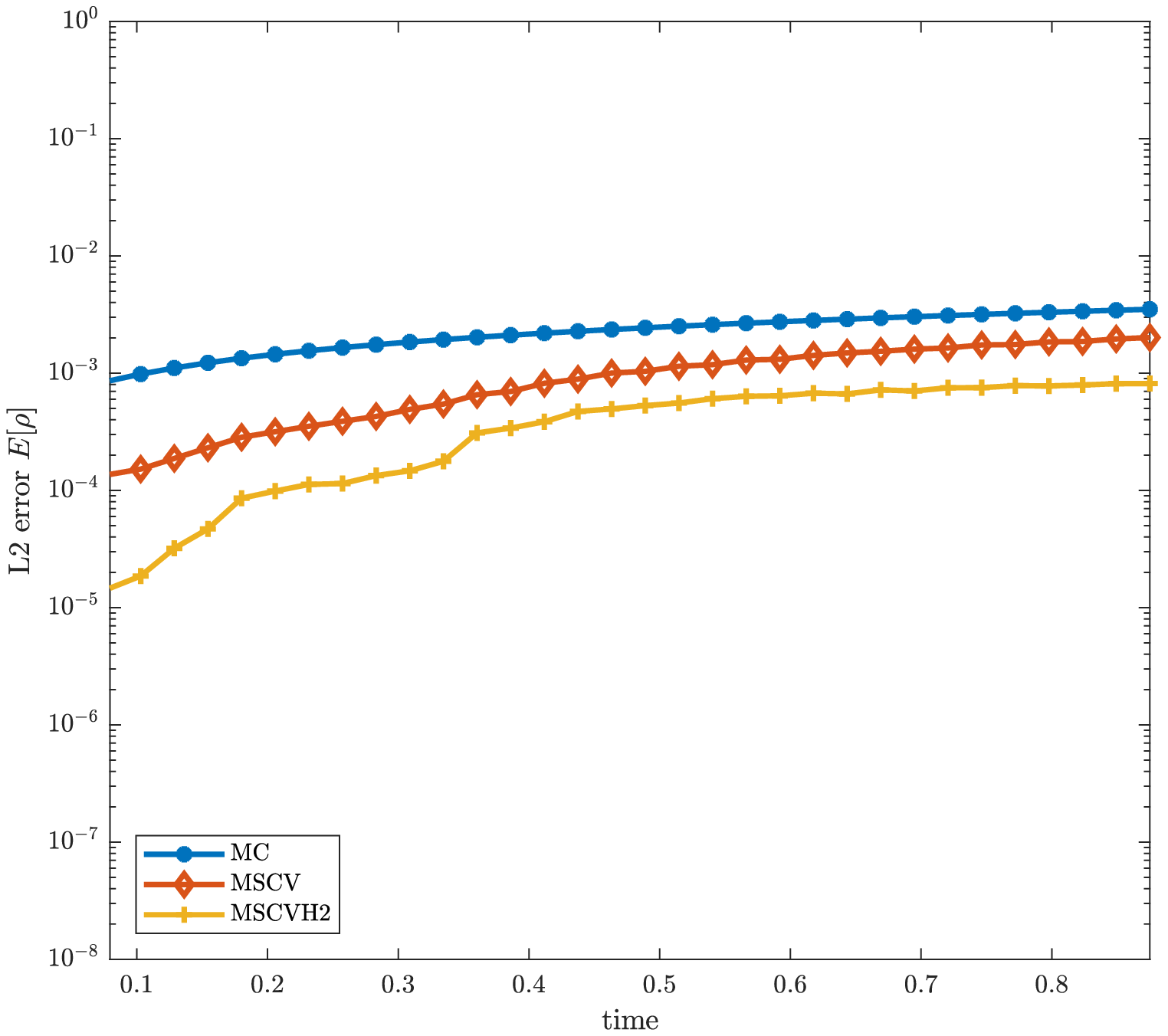}\\
		\includegraphics[width=.43\textwidth]{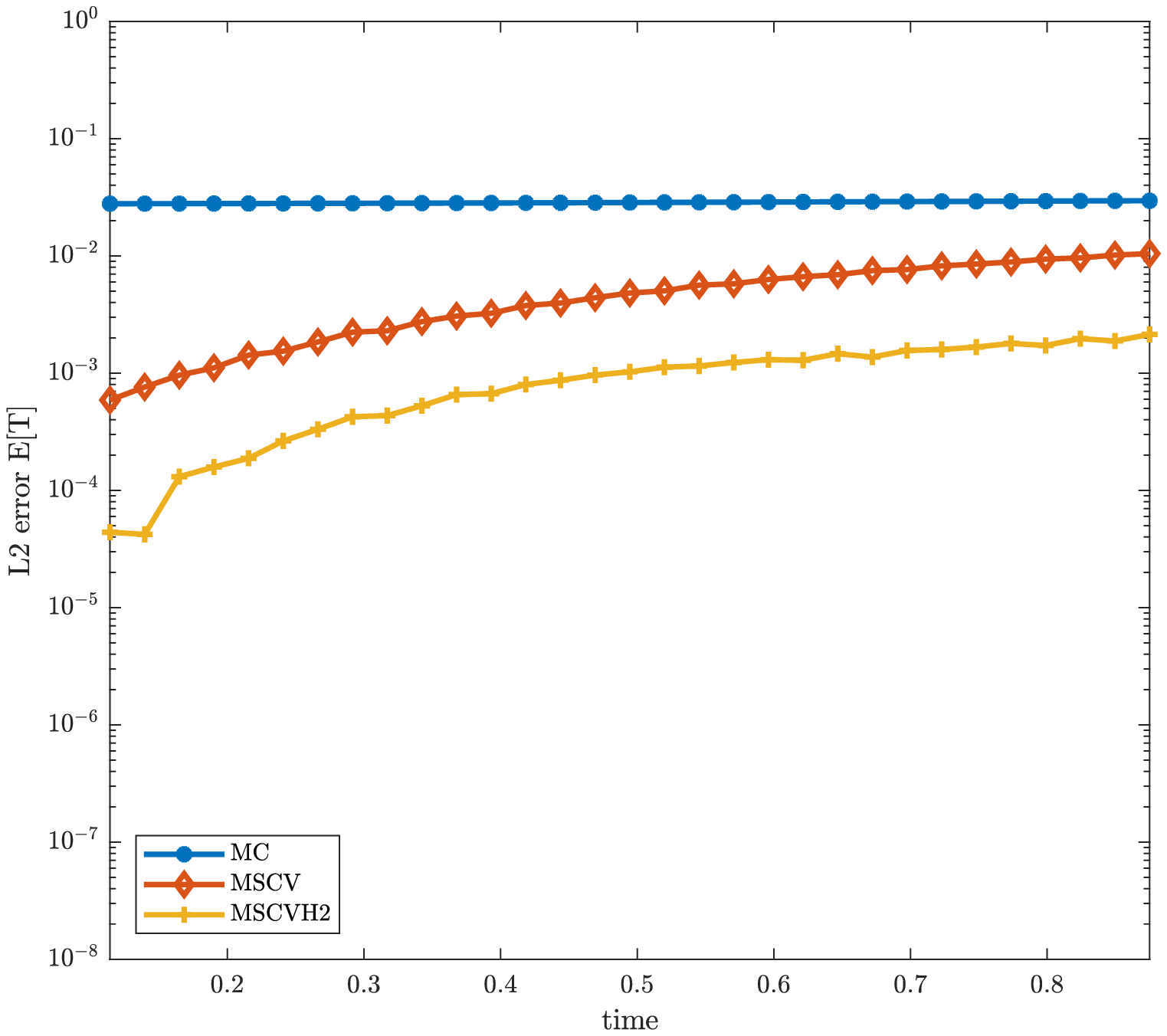}
		\includegraphics[width=.43\textwidth]{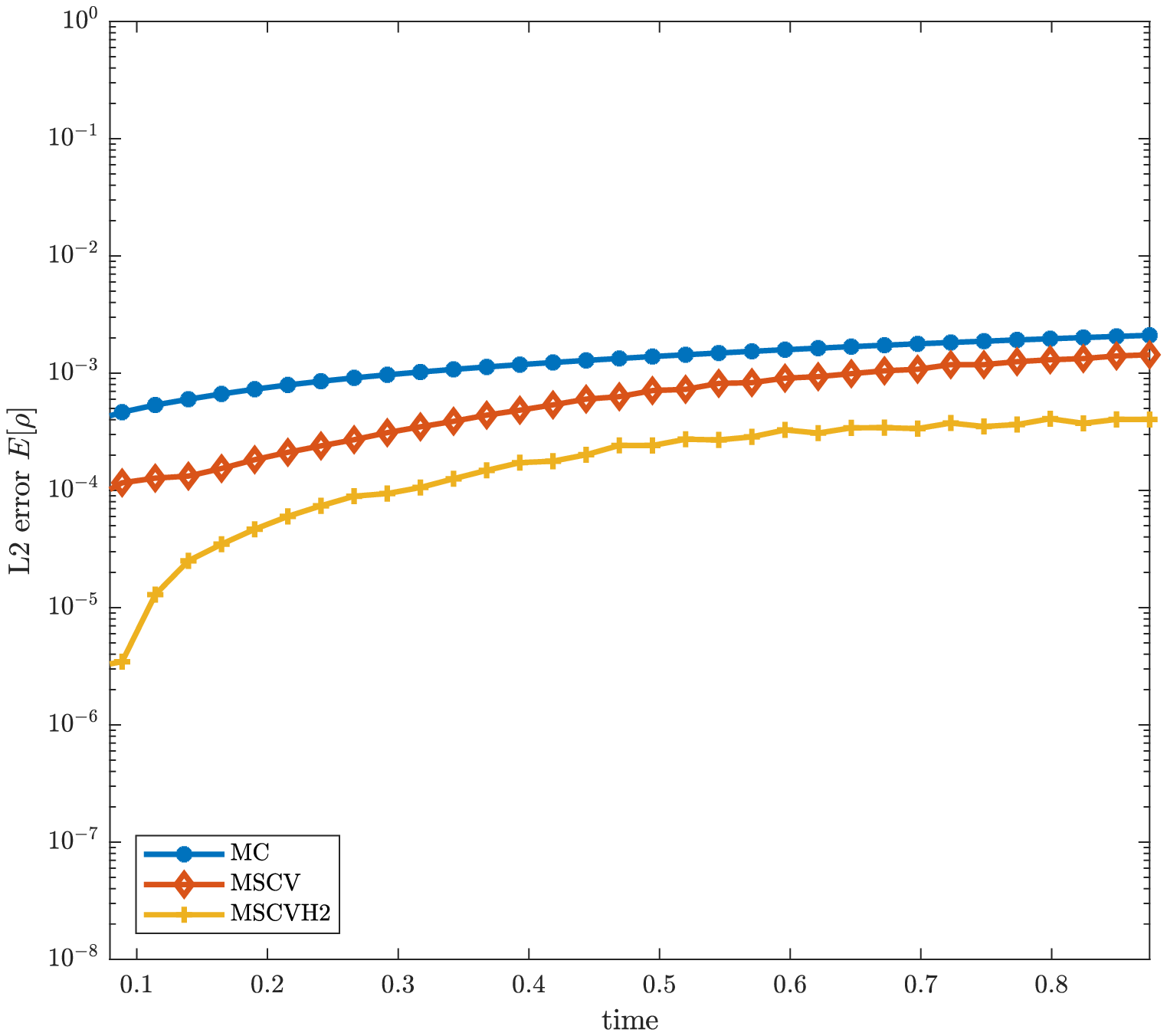}
		\caption{Test 2. Sod test with uncertainty in the initial data. $L_2$ norm of the error for the standard MC method, the MSCV method and the hierarchical multiple control variate MSCVH2 method for the expectation of the temperature (left) and for the density (right). The number of samples used for Boltzmann equation is $M=10$, for the BGK model is $M_{E_1}=100$ while for the compressible Euler system is $M_{E_2}=10^5$. Top: $\varepsilon=10^{-2}$. Middle: $\varepsilon=10^{-3}$. Bottom: $\varepsilon=2 \times \ 10^{-4}$. }\label{Figure6}
	\end{center}
\end{figure}

\subsubsection{Test 2. Sod test with uncertain initial data}
The initial conditions are
\bea
&\rho_0(x)=1, \ \ T_0(z,x)=1+sz \qquad &\textnormal{if} \ \ 0<x<L/2 \\
&\rho_0(x)=0.125,\ T_0(z,x)=0.8+sz  \qquad &\textnormal{if} \ \ L/2<x<1
\eea
with $s=0.25$, $z$ uniform in $[0,1]$ and equilibrium initial distribution
\[
f_0(z,x,\w)=\frac{\rho_0(x)}{2\pi } \exp\left({-\frac{|v|^2}{2T_0(z,x)}}\right).
\] 
The velocity space is truncated with $v_{\min}=v_{\max}=8$.

We perform three different computations corresponding to $ \varepsilon=10^{-2}$, $ \varepsilon=10^{-3}$ and $ \varepsilon=2 \times 10^{-4}$. The final time is fixed to $T_f=0.875$. 

In Figure \ref{Figure6}, we report the $L_2$ norms of the errors for the standard MC method, the MSCV approach and the hierarchical MSCVH2 method for the expected value of the temperature and the density as a function of time. The number of samples used for the BGK model is $M_{E_1}=100$ while for the compressible Euler system is $M_{E_2}=10^5$. In all regimes the gain of the hierarchical approach is remarkable and improves for smaller values of the Knudsen number.

Next, we discuss the numerical results of the two multi-scale control variate approach of Section \ref{sec:twoc}. For brevity, we report the results only for $\varepsilon=5 \times 10^{-4}$. 
To illustrate the models behavior, in Figure \ref{Figure9}, we show the expectation for the temperature and the density obtained with the Boltzmann model and the ones obtained with the two different choices of the relaxation frequencies for the BGK model.
\begin{figure}
	\begin{center}
		\includegraphics[width=0.43\textwidth]{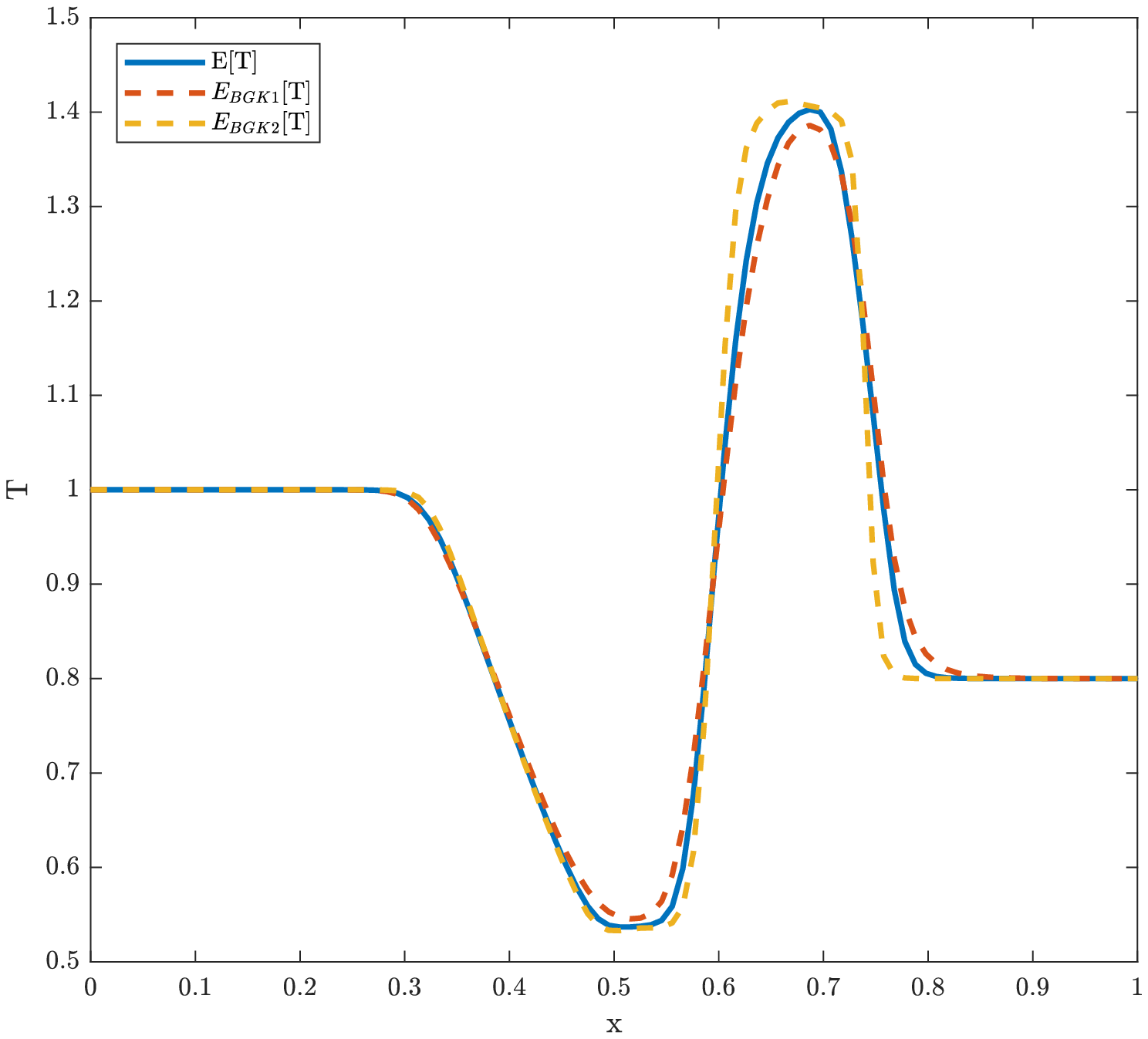}
		\includegraphics[width=0.43\textwidth]{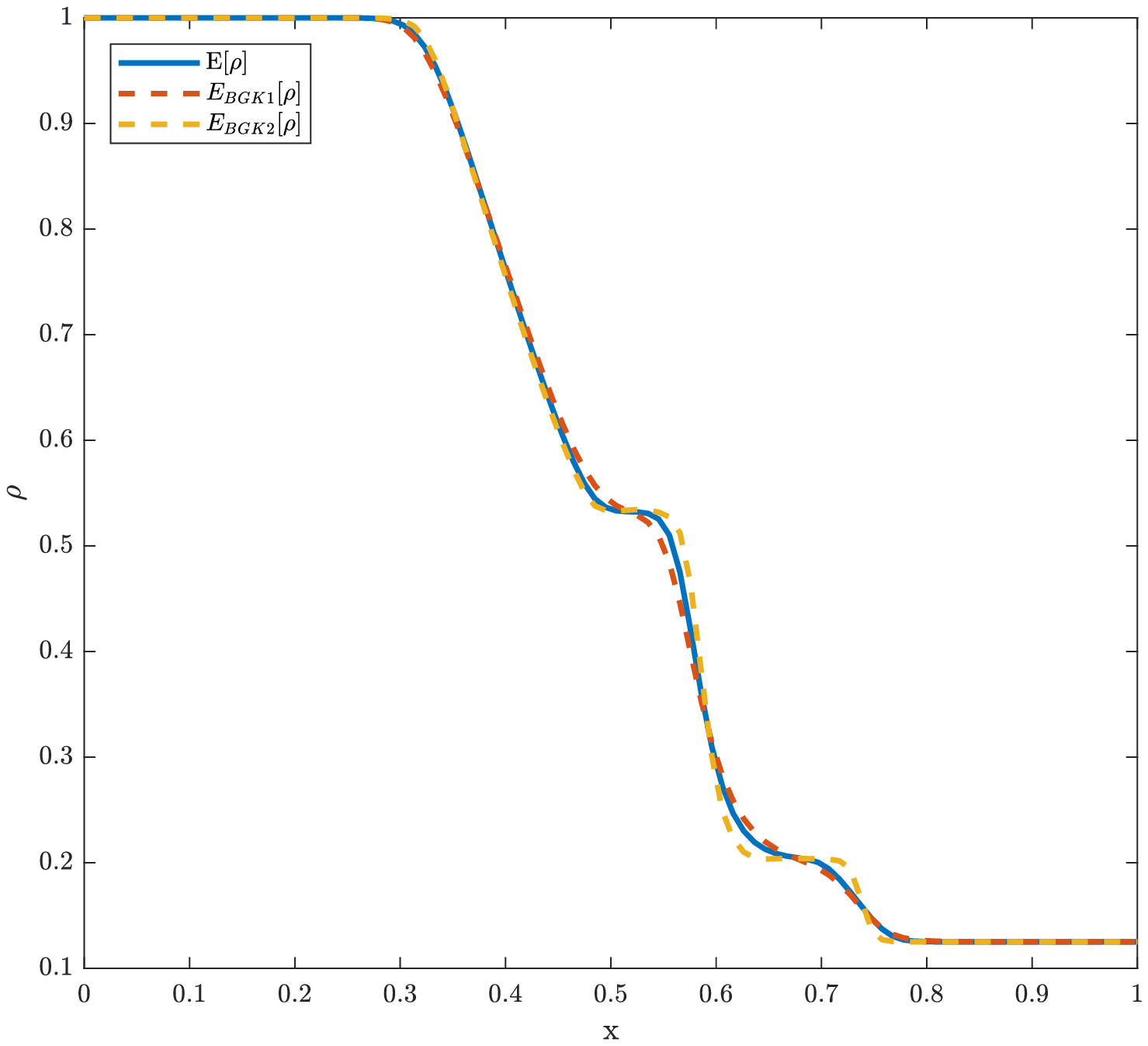}\\
		\caption{Test 2. Sod test with uncertainty in the initial data. Temperature profile at final time on the left. Density profile at final time on the right. Expectation for the Boltzmann model and the two BGK models with $\nu=\rho$ and $\nu=0.125\rho$.}
		\label{Figure9}
	\end{center}
\end{figure}
In Figure \ref{Figure10}, we report the $L_2$ norms of the errors for the expected value of the temperature and the density as a function of time. We plot the error for the standard MC method, for the single MSCV methods with the two different collision frequencies and the multiple MSCV2 method. The number of samples used is $M_E=1000$ on the top, $M_E=5000$ in the middle and $M_E=10000$ on the bottom. The effective gain obtained with the MSCV2 ca be easily observed. 
 
\begin{figure}
	\begin{center}
		\includegraphics[width=.43\textwidth]{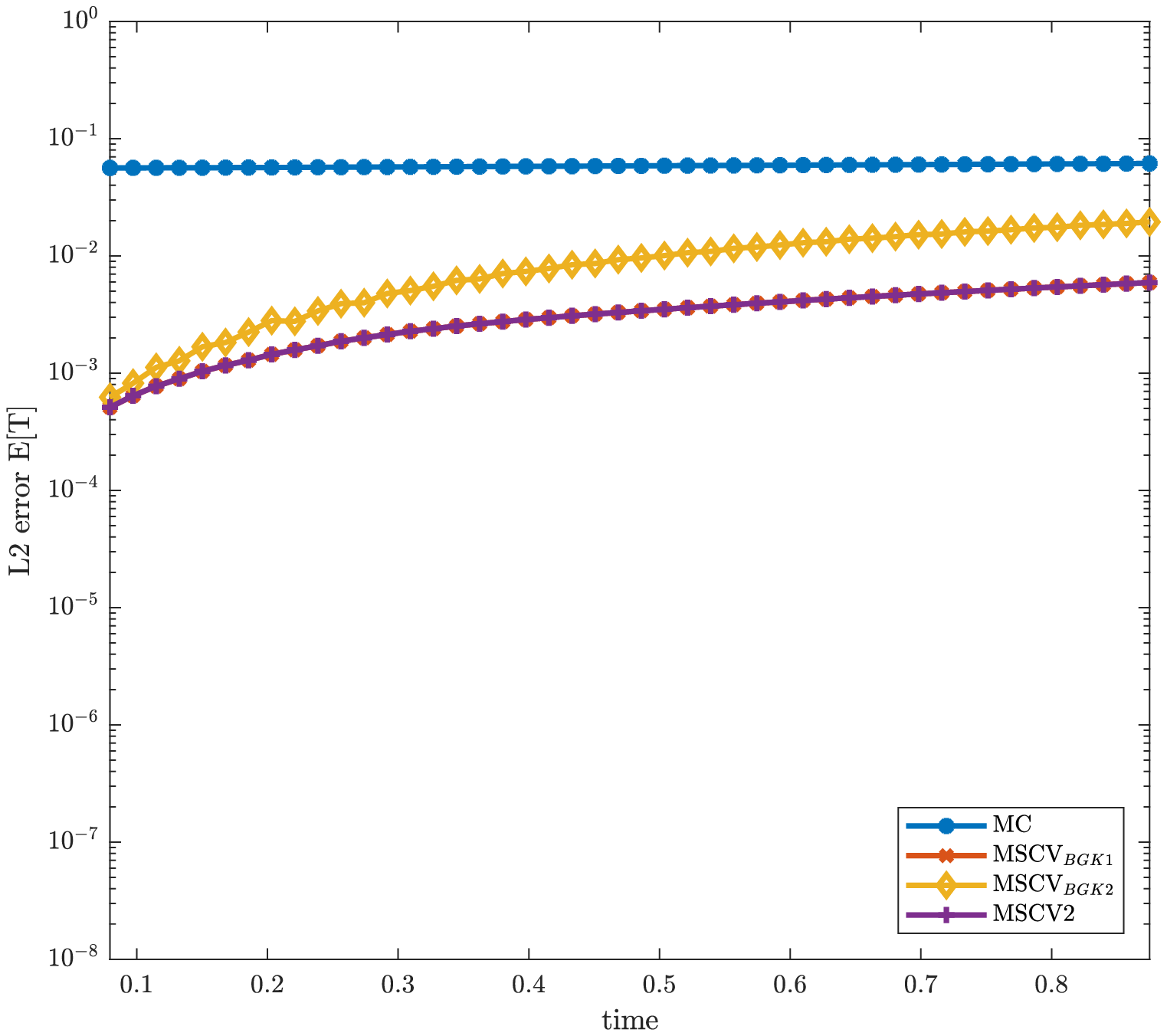}
		\includegraphics[width=.43\textwidth]{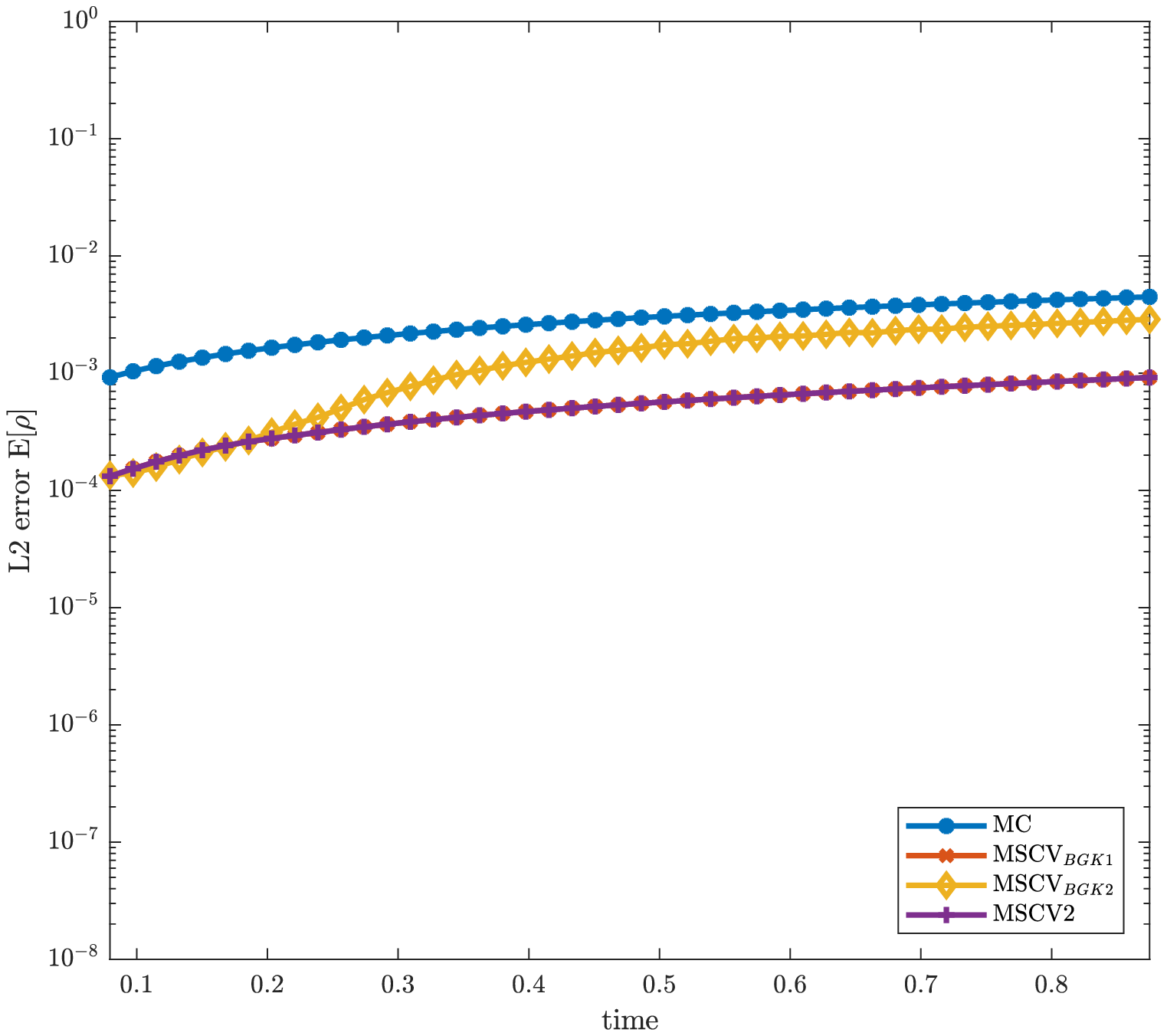}\\
			\includegraphics[width=.43\textwidth]{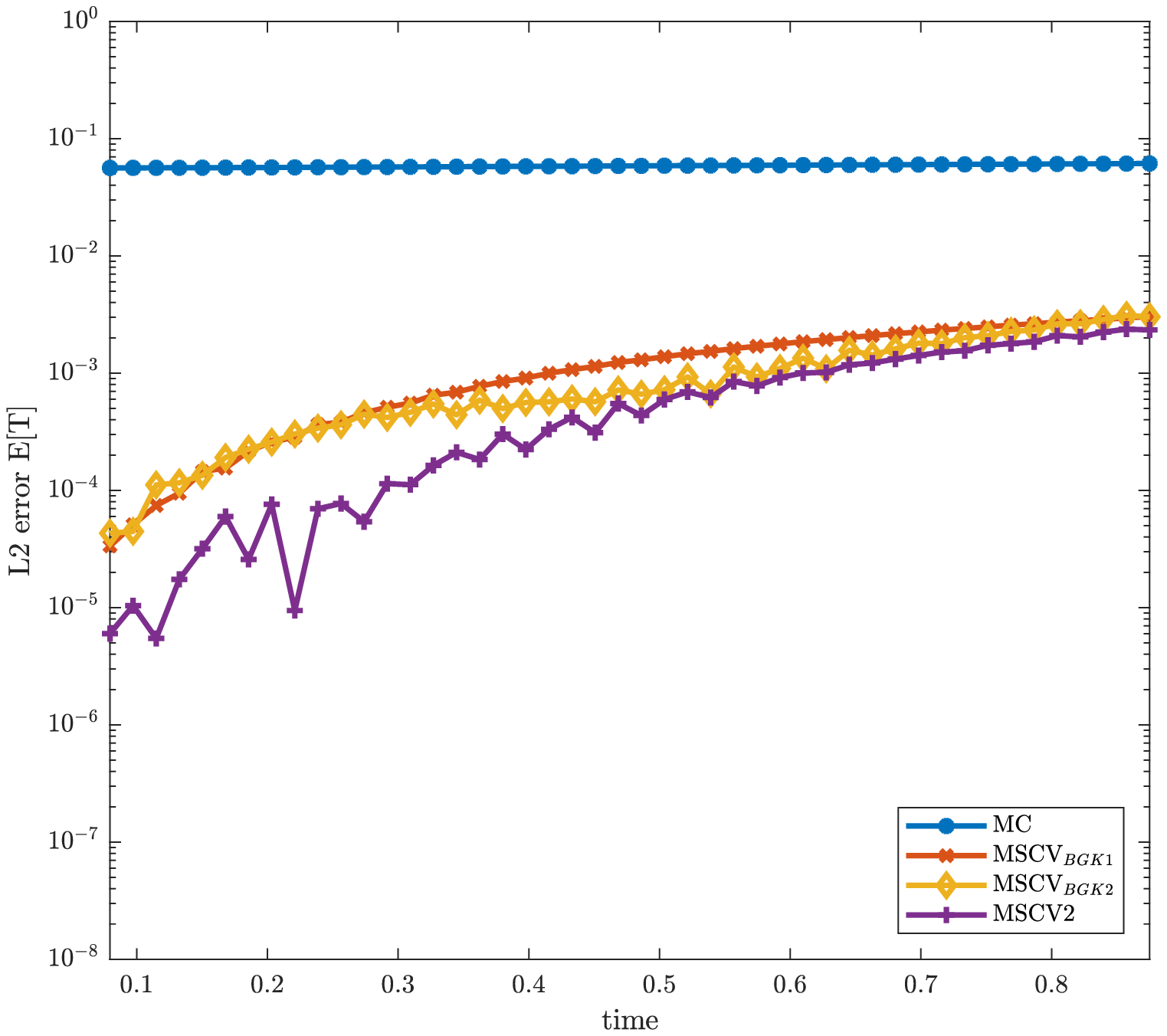}
		\includegraphics[width=.43\textwidth]{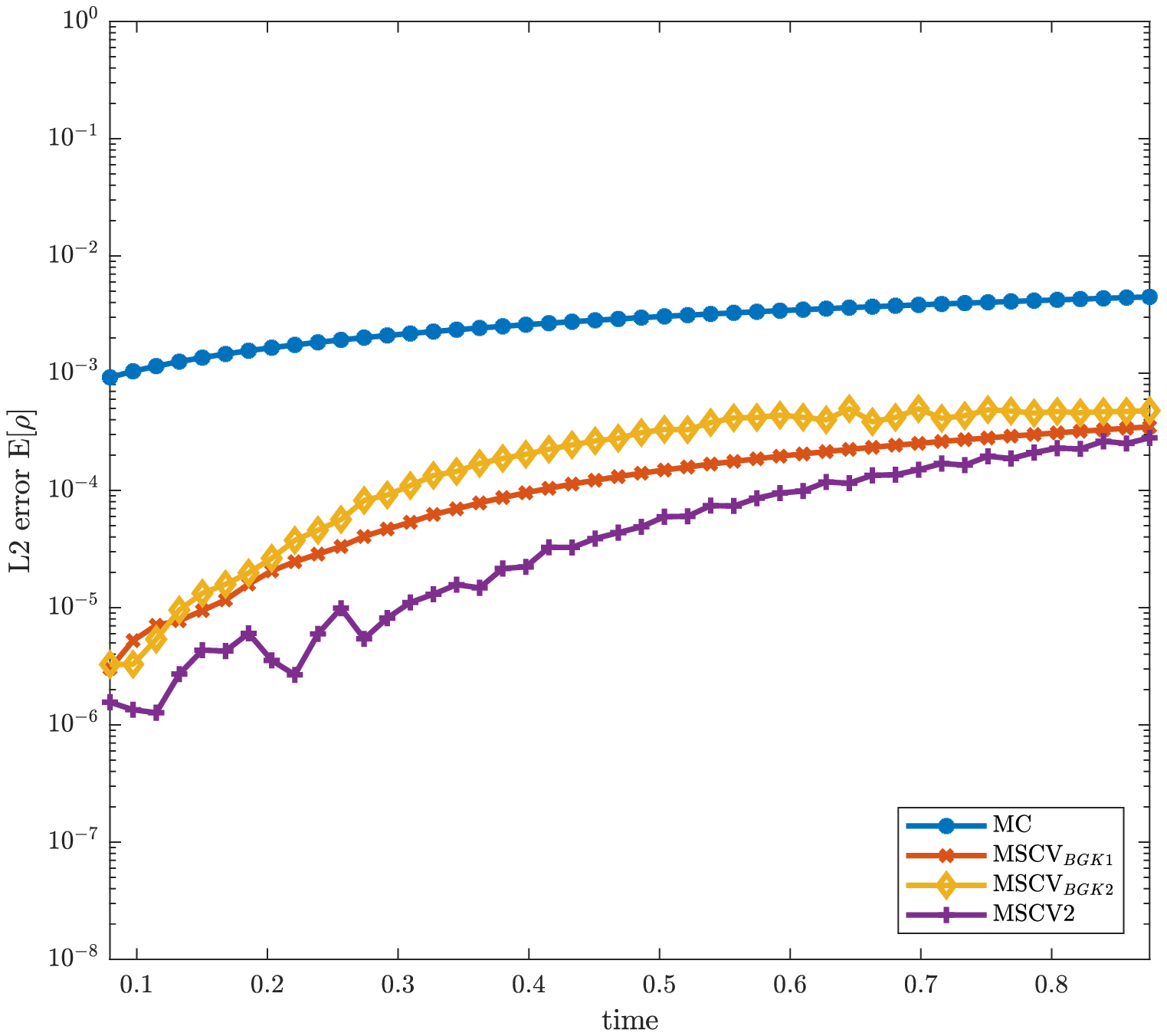}\\
			\includegraphics[width=.43\textwidth]{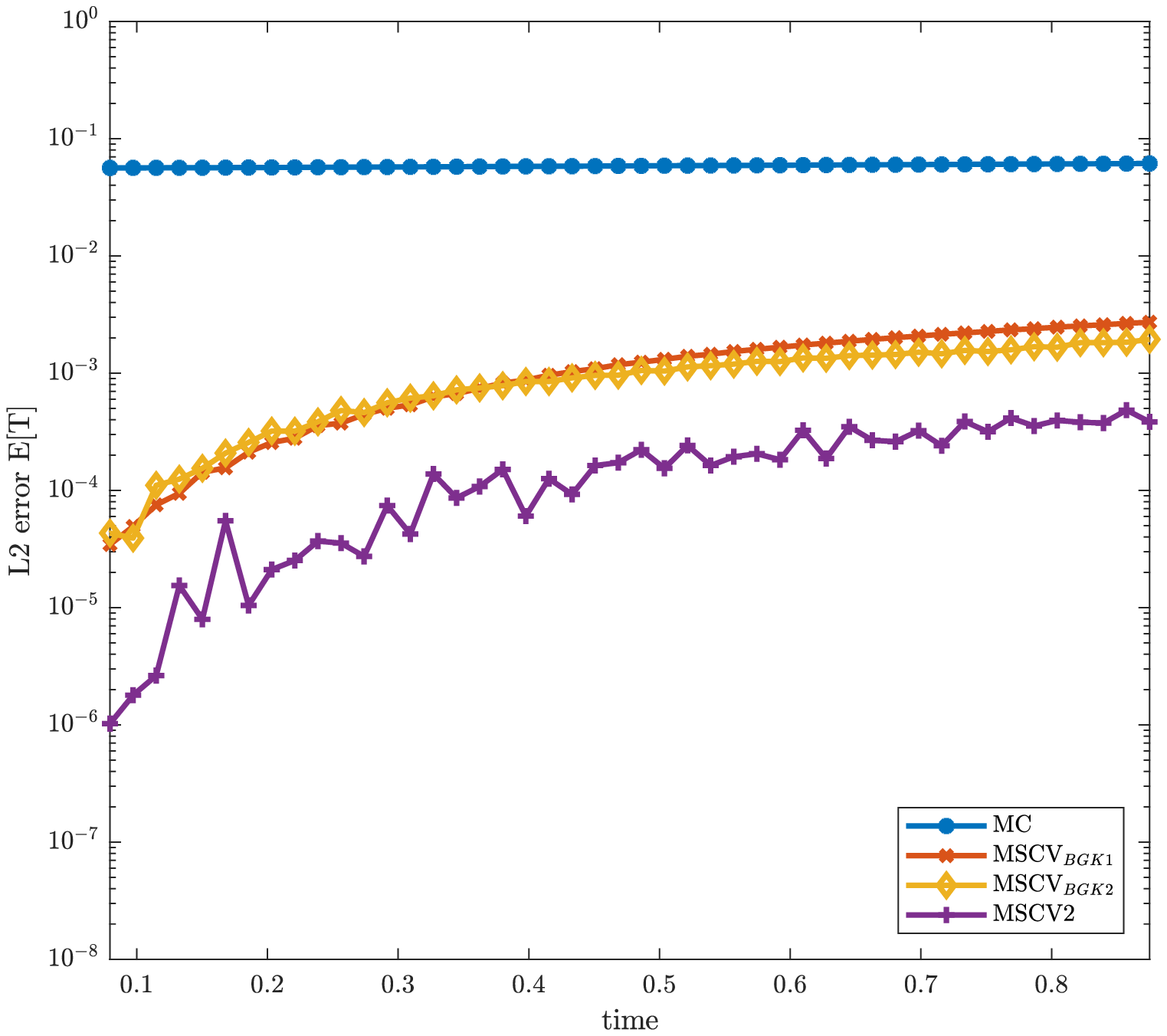}
		\includegraphics[width=.43\textwidth]{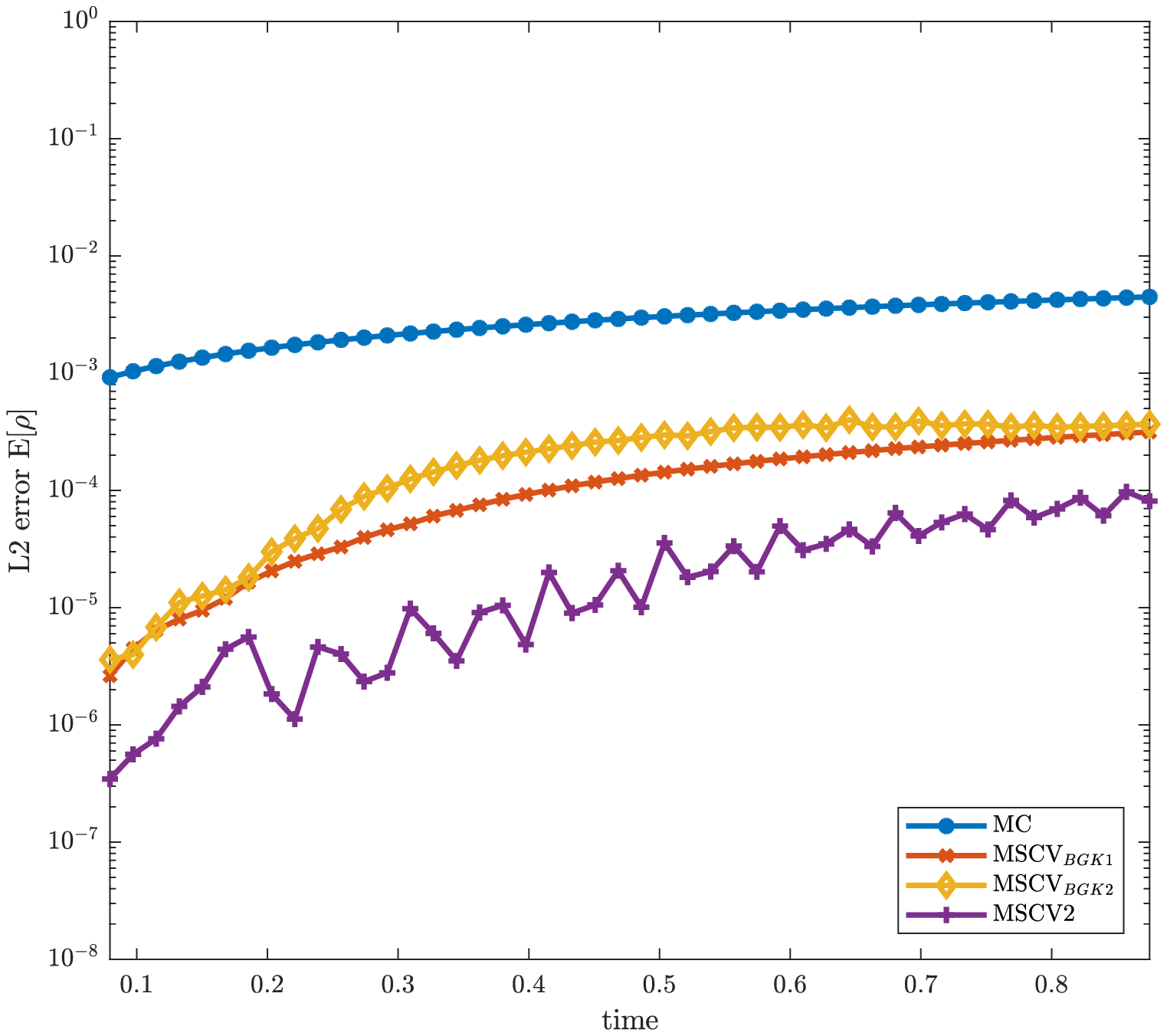}
		\caption{Test 2. Sod test with uncertainty in the initial data. $L_2$ norm of the error for the standard MC method, the MSCV method and the multiple control variate MSCV2 method for the expectation of the temperature (left) and for the density (right). The number of samples used for Boltzmann equation is $M=10$. Top: $M_E=10^3$ points. Middle: $M_E=5\times 10^3$ points. Bottom: $M_E=10^5$ points. }\label{Figure10}
	\end{center}
\end{figure}

\subsubsection{Test 3. Sudden heating problem with uncertain boundary condition}
In the last test case, the initial condition is a constant state in space given by
\be
f_0(x,\w)=\frac{1}{2\pi T^0}e^{-\dfrac{\w^2}{2T^0}}, \ T^0=1, \qquad x\in[0,1].
\ee 
At time $t=0$, the temperature at the left wall suddenly changes and it starts to heat the gas. We assume diffusive equilibrium boundary conditions and uncertainty on the wall temperature:
\be
T_w(z)=2(T^0+sz), \ s=0.2,
\ee
and $z$ uniform in $[0,1]$.
The velocity space is truncated as before with $v_{\min}=v_{\max}=8$, the time discretization and the time step are the same of the previous test. We perform again three different computations corresponding to $ \varepsilon=10^{-2}$, $ \varepsilon=10^{-3}$ and $ \varepsilon=2 \times 10^{-4}$. The final time is fixed to $T_f=0.8$.

In Figure \ref{Figure13}, the $L_2$ norms of the errors for the expected value of the temperature and the density are given as a function of time. The error curves refer to the standard MC method, the MSCV approach when the BGK control variate is employed and the hierarchical multiple MSCVH2 method with BGK and compressible Euler models as control variates. The number of samples for the Boltzmann model is $M=10$, for the BGK model is $M_{E_1}=100$ while for the compressible Euler is $M_{E_2}=10^5$. We see again that the hierarchical approach improves the computed values of the expectations in all regimes. The gain is stronger for the density and tends to increase for smaller values of the Knudsen number.
\begin{figure}
	\begin{center}
		\includegraphics[width=.43\textwidth]{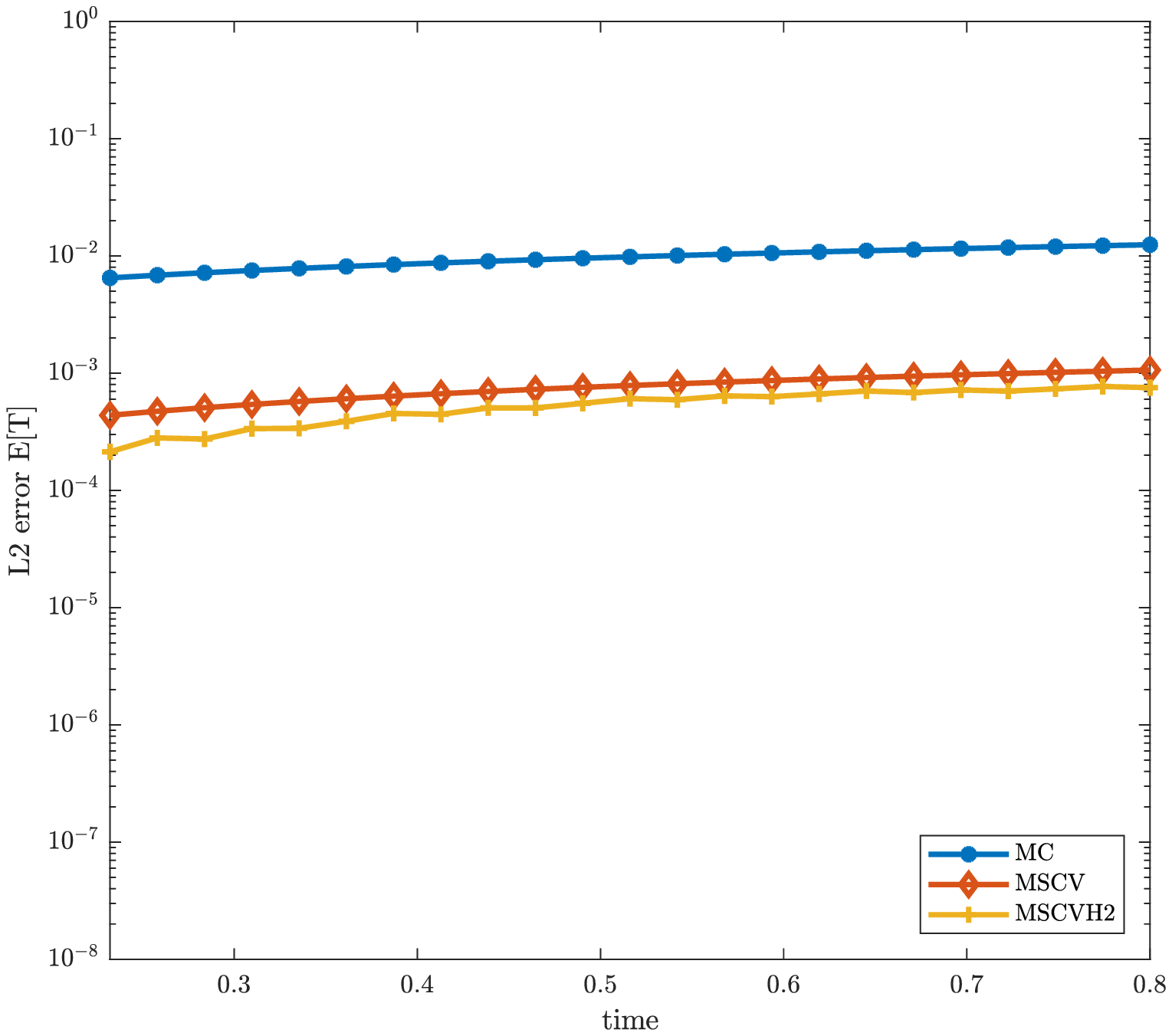}
		\includegraphics[width=.43\textwidth]{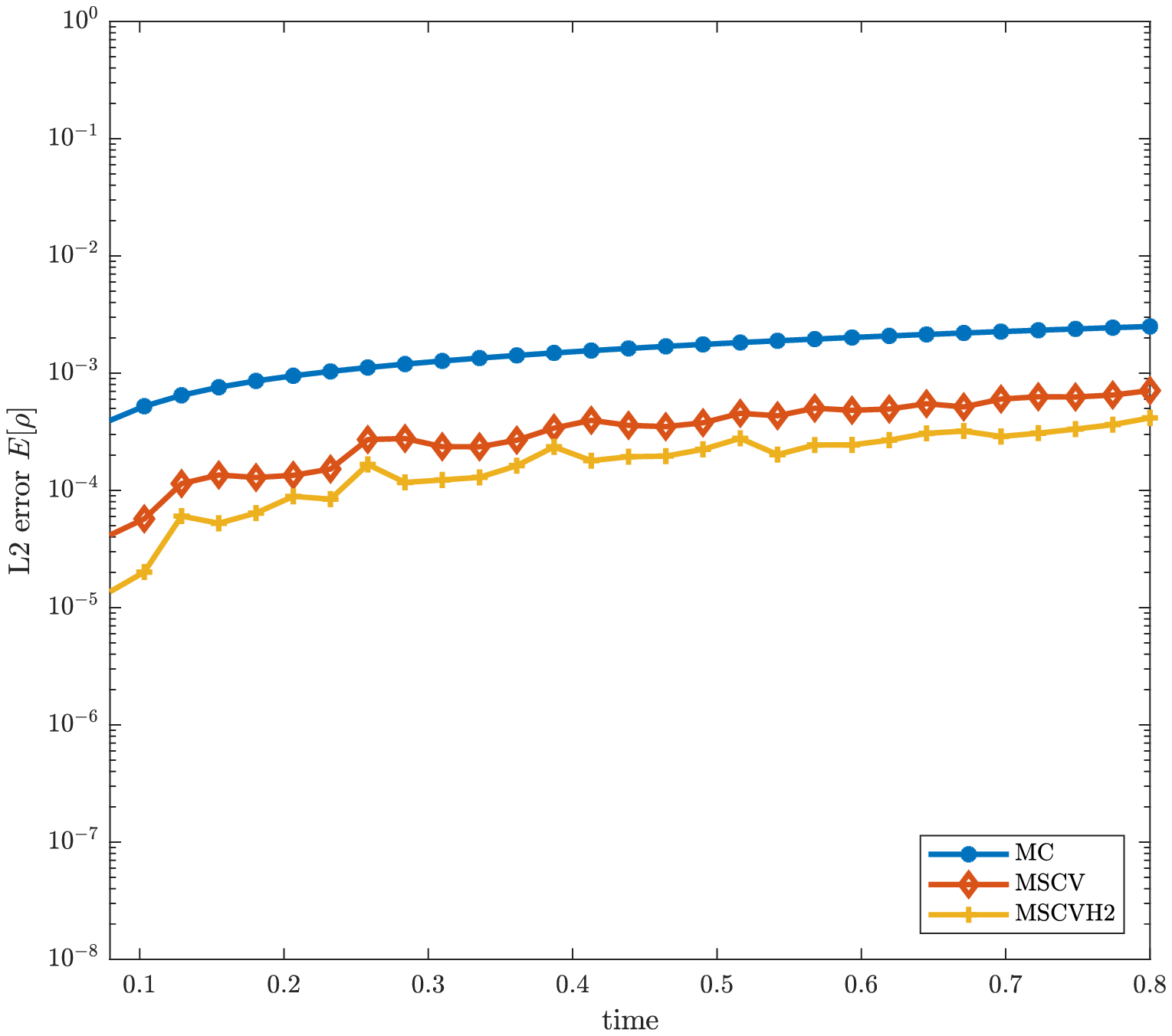}\\
			\includegraphics[width=.43\textwidth]{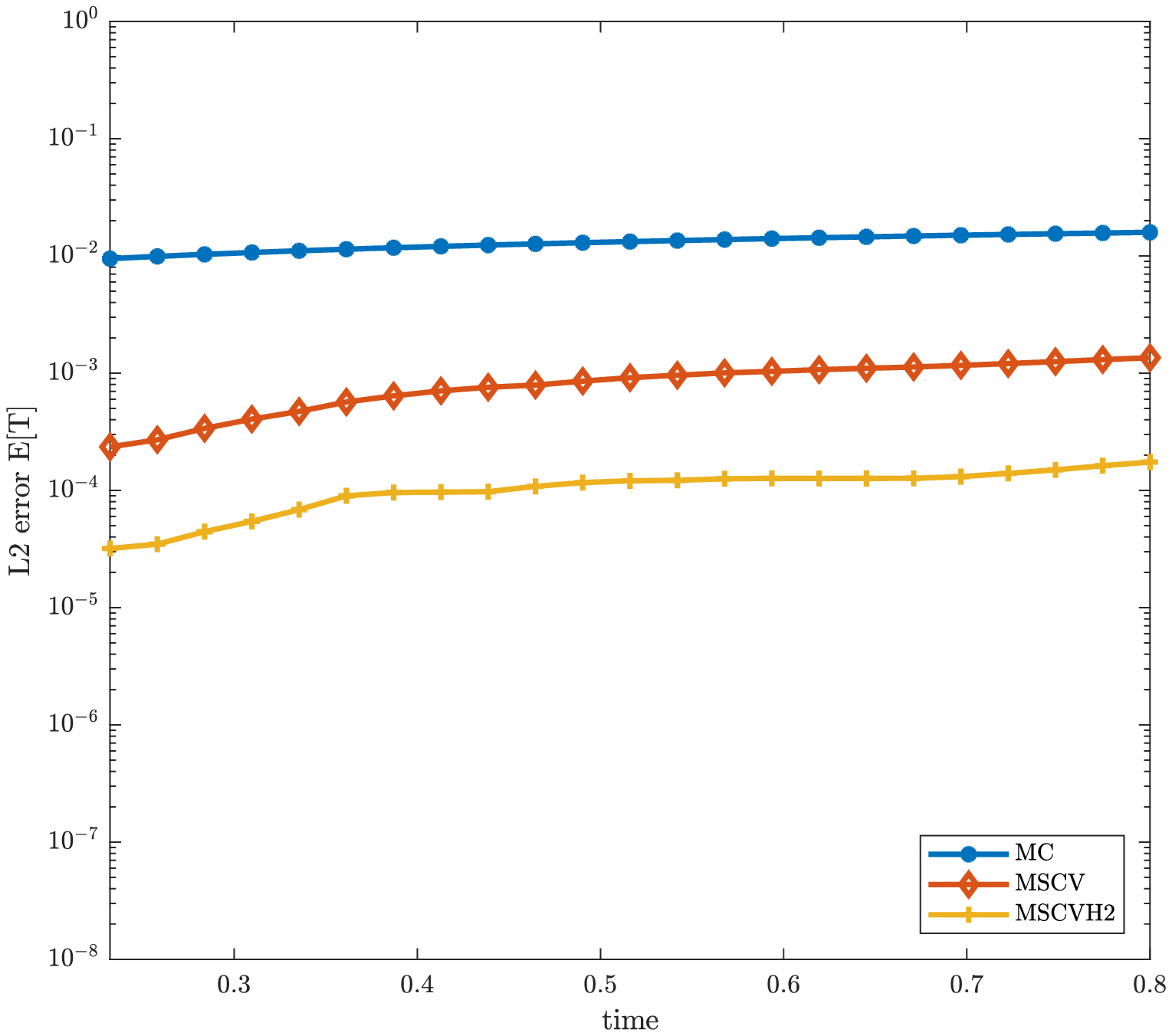}
		\includegraphics[width=.43\textwidth]{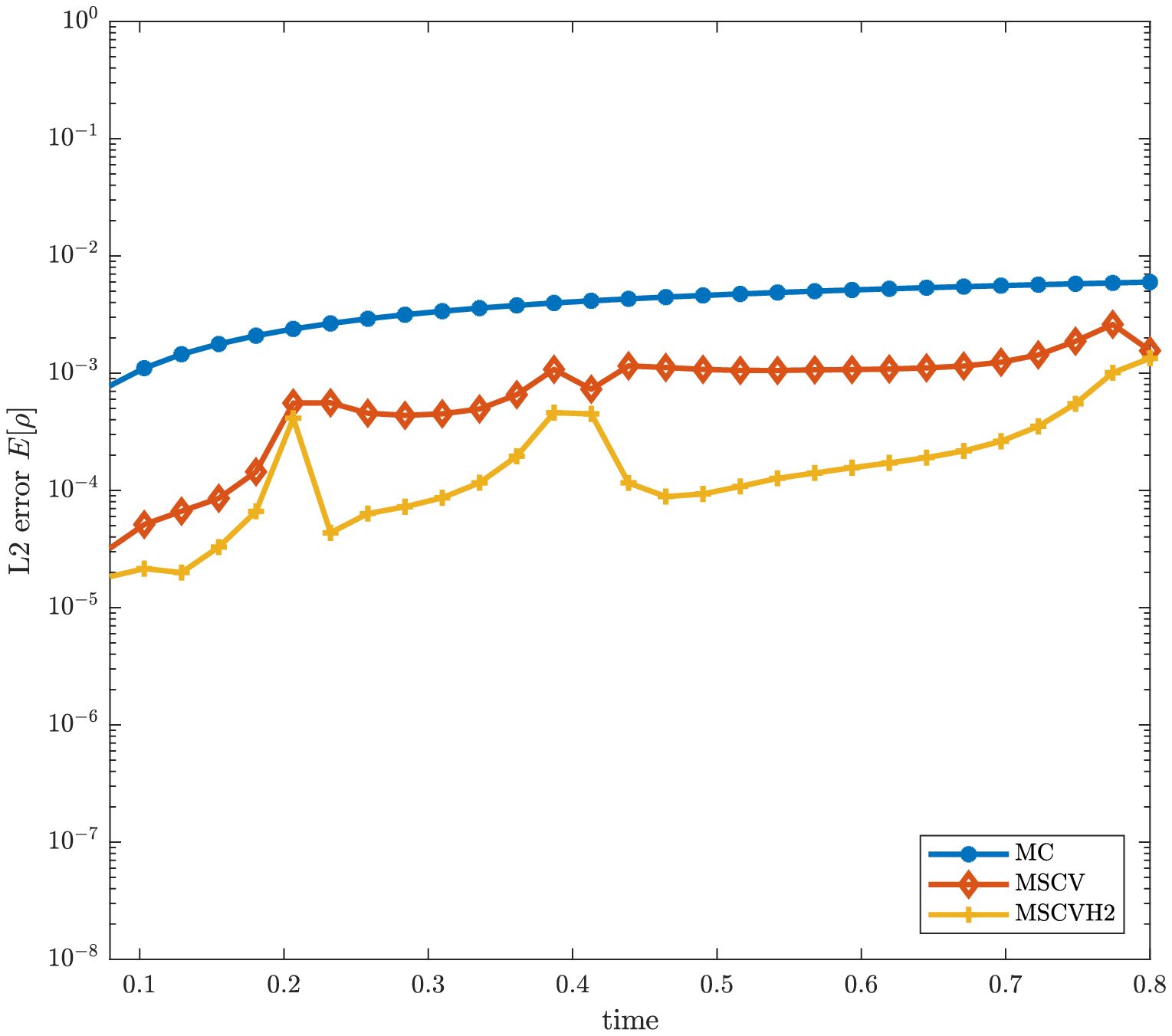}\\
			\includegraphics[width=.43\textwidth]{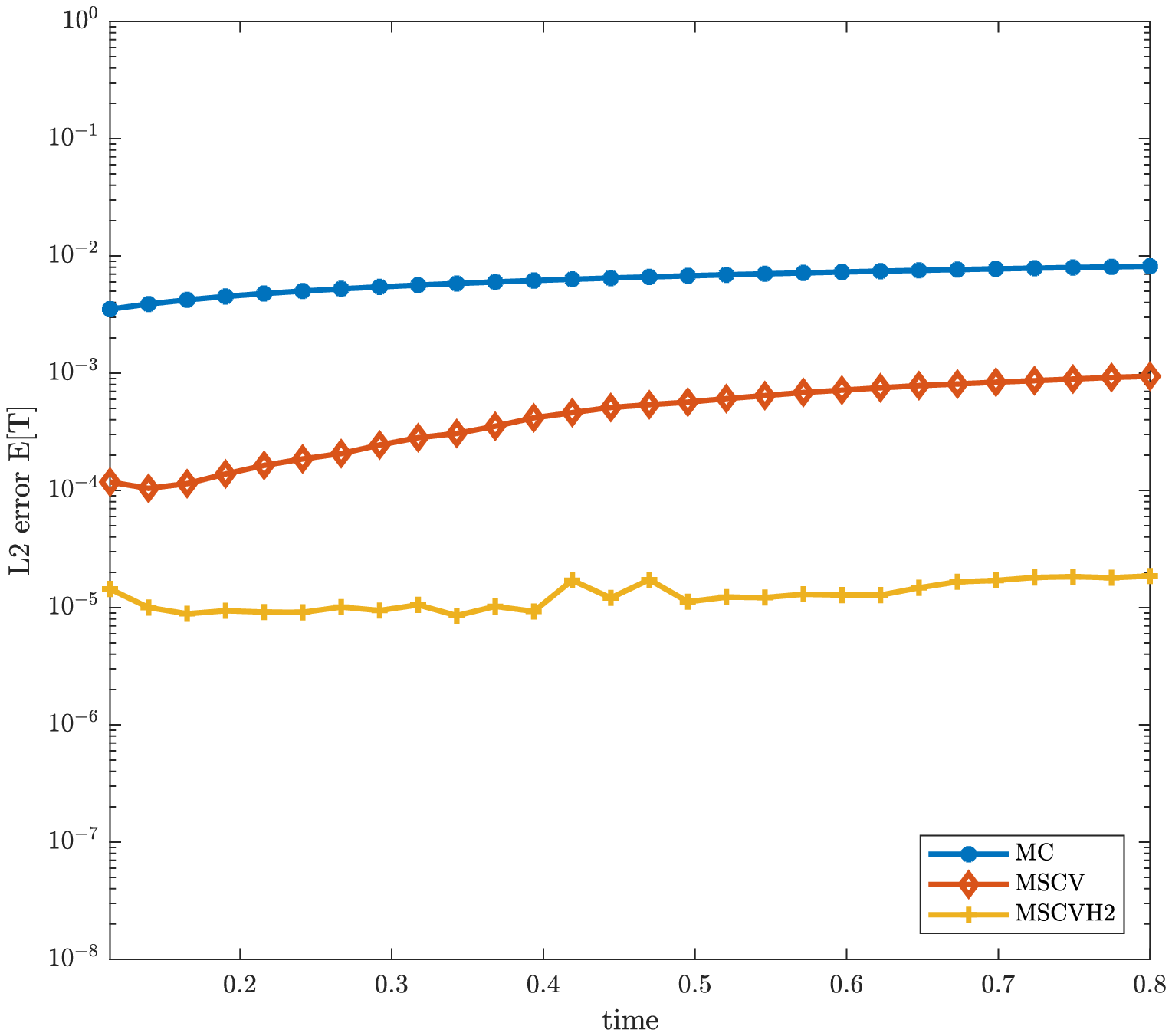}
		\includegraphics[width=.43\textwidth]{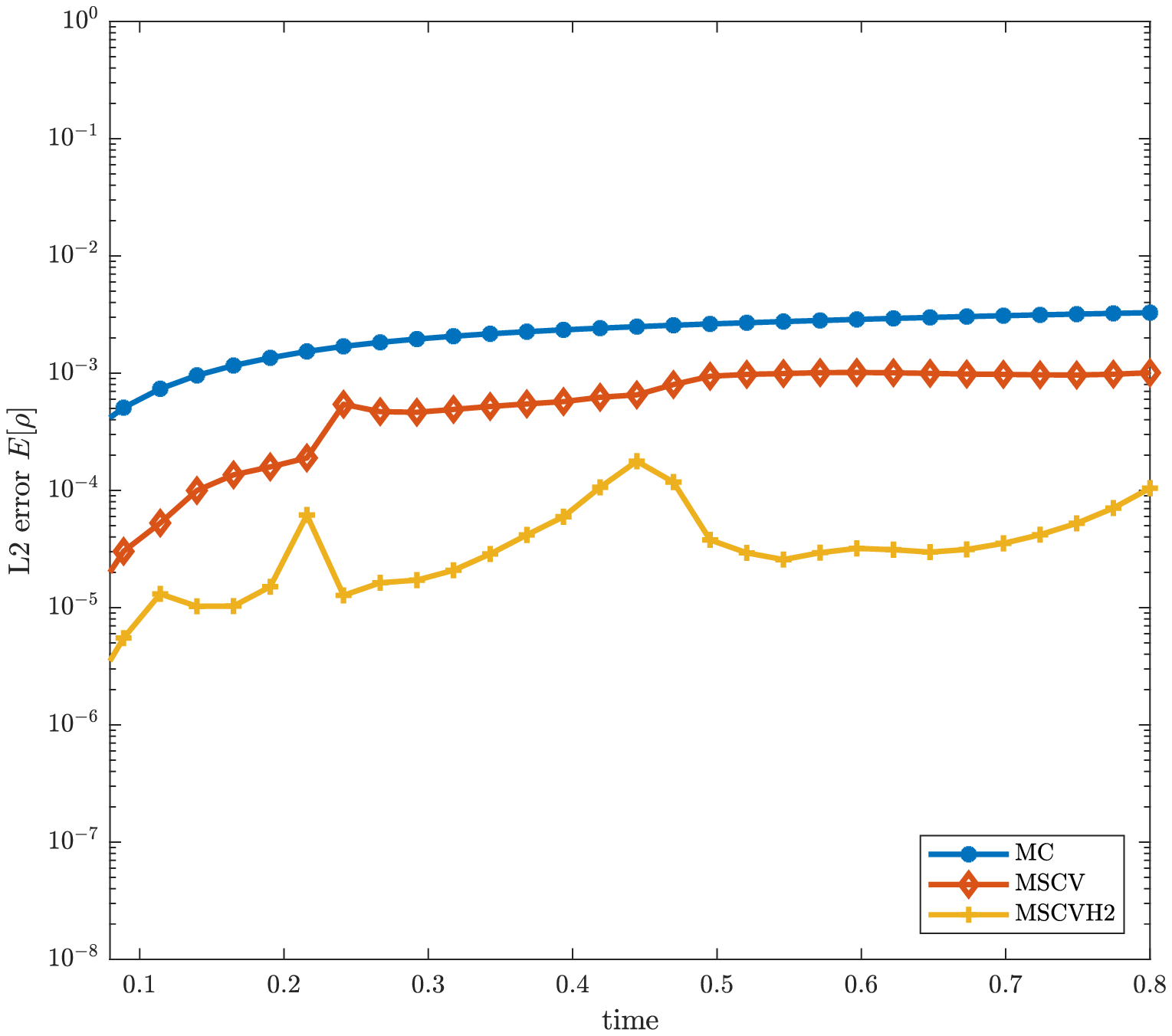}\\
		\caption{Test 3. Sudden heating problem with uncertainty in the boundary condition. $L_2$ norm of the error for the standard MC method, the MSCV method and the hierarchical multiple control variate MSCVH2 method for the expectation of the temperature (left) and for the density (right). The number of samples used for Boltzmann equation is $M=10$, for the BGK model is $M_{E_1}=100$ while for the compressible Euler system is $M_{E_2}=10^5$. Top: $\varepsilon=10^{-2}$. Middle: $\varepsilon=10^{-3}$. Bottom: $\varepsilon=2 \times \ 10^{-4}$.}\label{Figure13}
	\end{center}
\end{figure}

%%%%%%%%%%%%%%%%%%%%%%%%%%%%%%%%%%%%%%%%%%%%%%%%%%%%%%%%%%%%%%%%%%%%%%%%%%%%%%%%%%%%%%%%%%%%%%%%%%%%%%%%%%%%%%%%%%%%%5

We finally discuss the numerical results of the two multi-scale control variate approach of Section \ref{sec:twoc} for the sudden heating problem. Again, for brevity, we report the results only for $\varepsilon=5 \times 10^{-4}$.  In Figure \ref{Figure16}, we show the expectation for the temperature on the left and the density on the right obtained with the Boltzmann model and the two different BGK models.
\begin{figure}
	\begin{center}
		\includegraphics[width=0.43\textwidth]{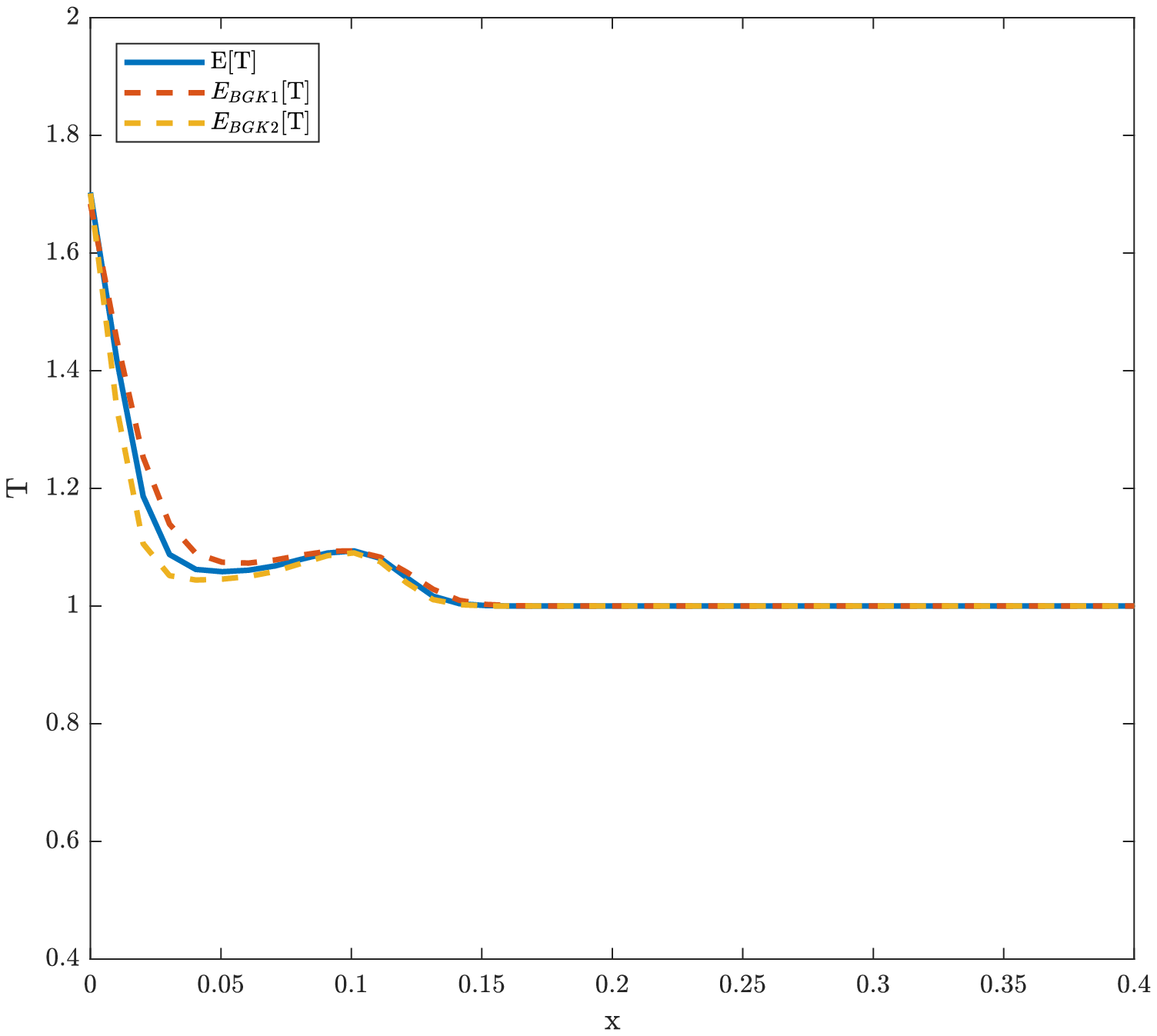}
		\includegraphics[width=0.43\textwidth]{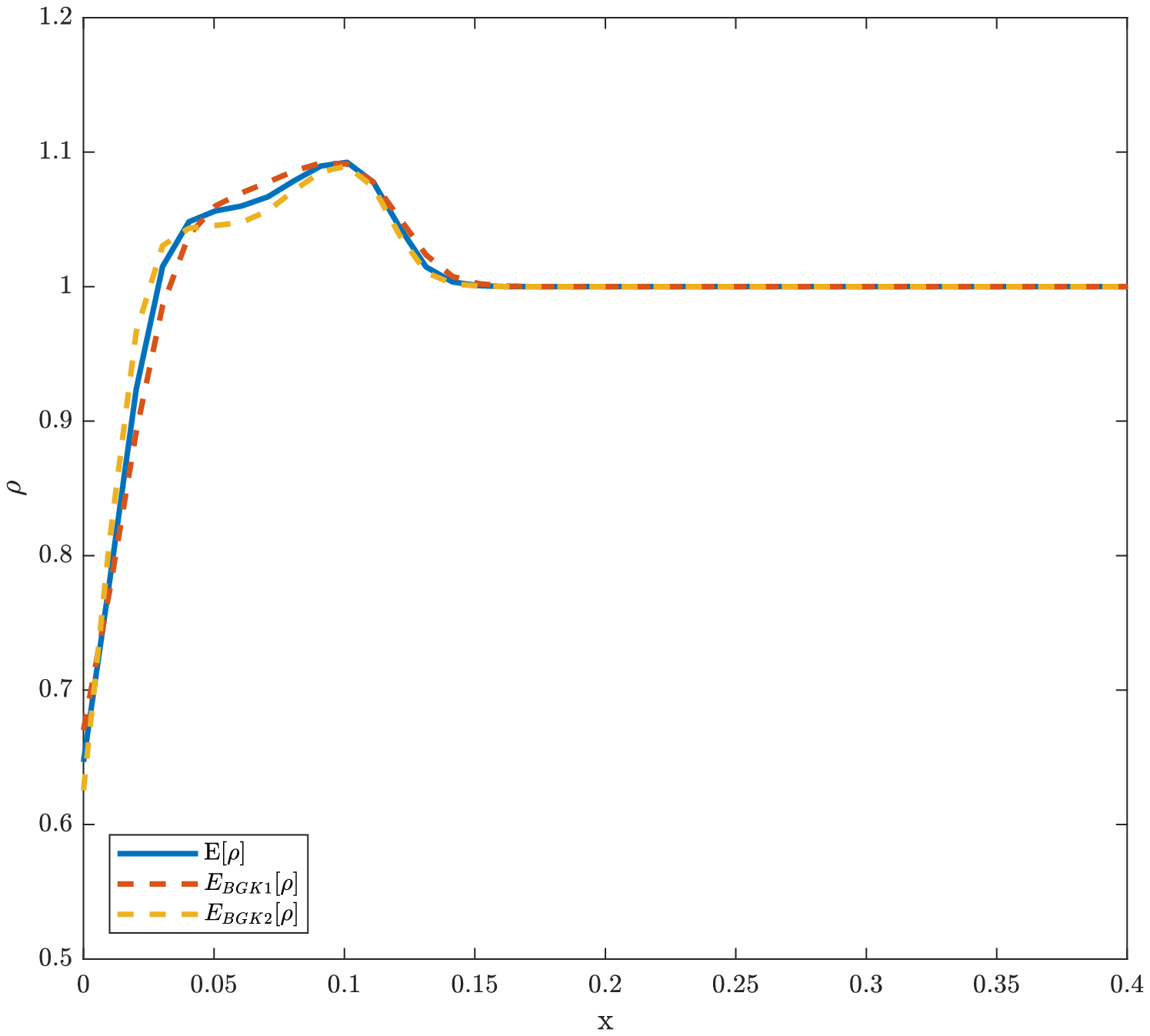}\\
		\caption{Test 3. Sudden heating problem with uncertainty in the boundary condition. Temperature profile at final time on the left. Density profile at final time on the right. Expectation for the Boltzmann model and the two BGK models with $\nu=\rho$ and $\nu=0.125\rho$.}
		\label{Figure16}
	\end{center}
\end{figure}
In Figure \ref{Figure17}, we report the $L_2$ norms of the errors for the expected value of the temperature   and for the density as a function of time. In the images, it is shown the error for the standard MC method, for the MSCV methods with the two different collision frequencies and the two multi-scale control variate MSCV2 method. The number of samples used is $M_E=1000$ on the top, $M_E=5000$ in the middle and $M_E=10000$ on the bottom. As shown, even for this test case, the MSCV2 method is able to improve the accuracy of the estimate provided that the control variate model solutions are evaluated with enough samples. Of course, more precise estimates of the relaxation rates in the BGK models would lead to even stronger improvements.

\begin{figure}
	\begin{center}
		\includegraphics[width=.43\textwidth]{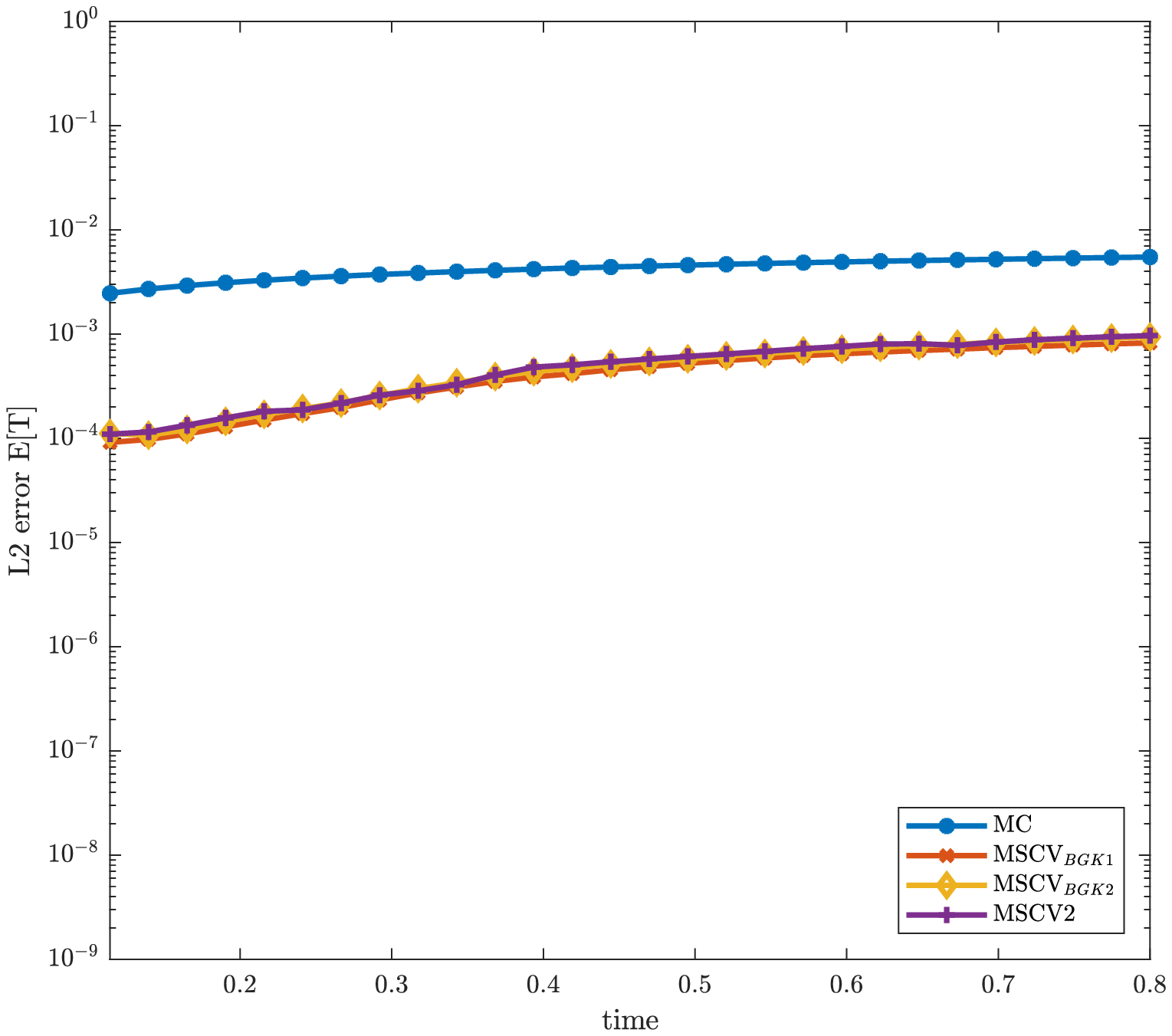}
		\includegraphics[width=.43\textwidth]{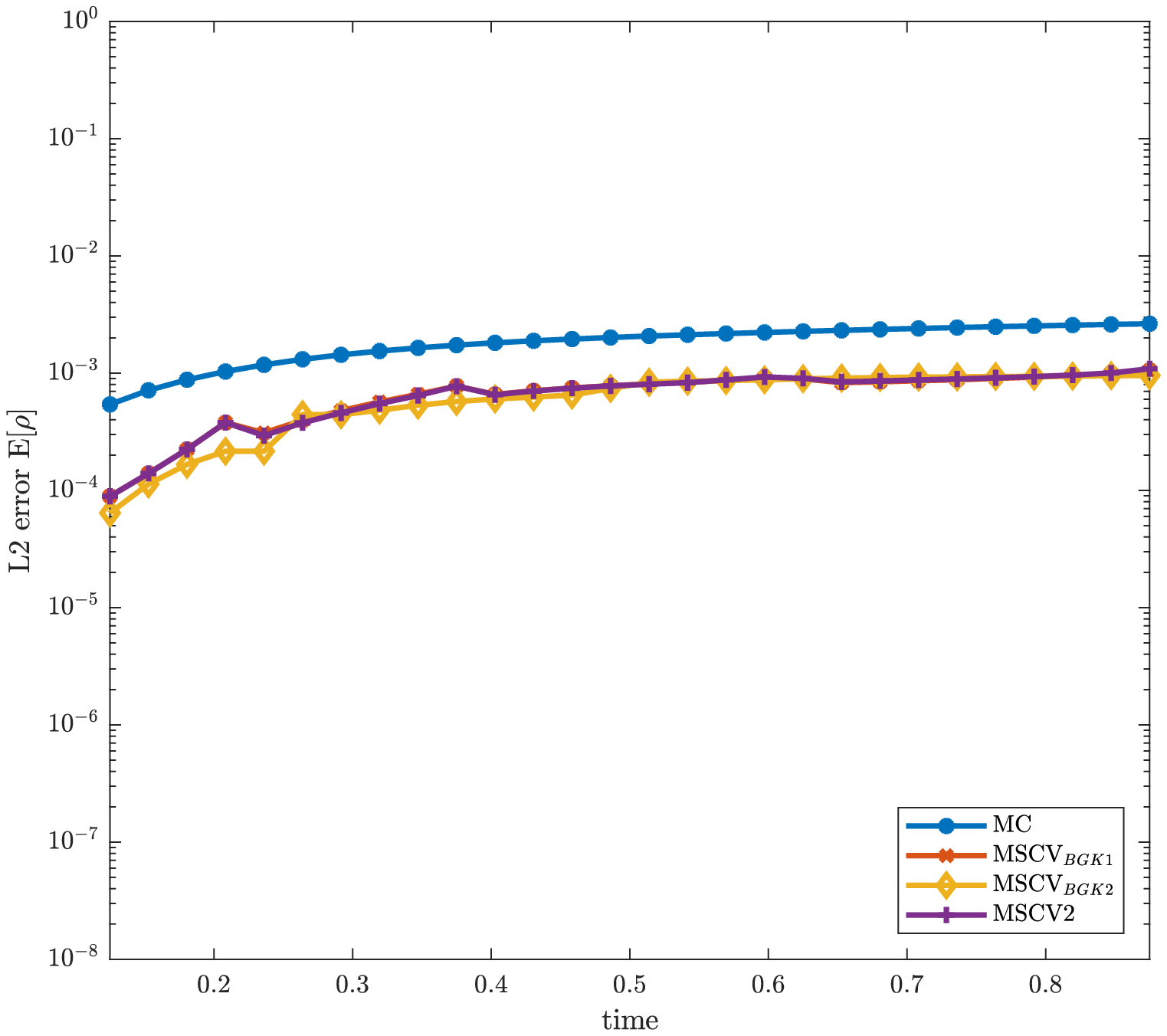}\\
		\includegraphics[width=.43\textwidth]{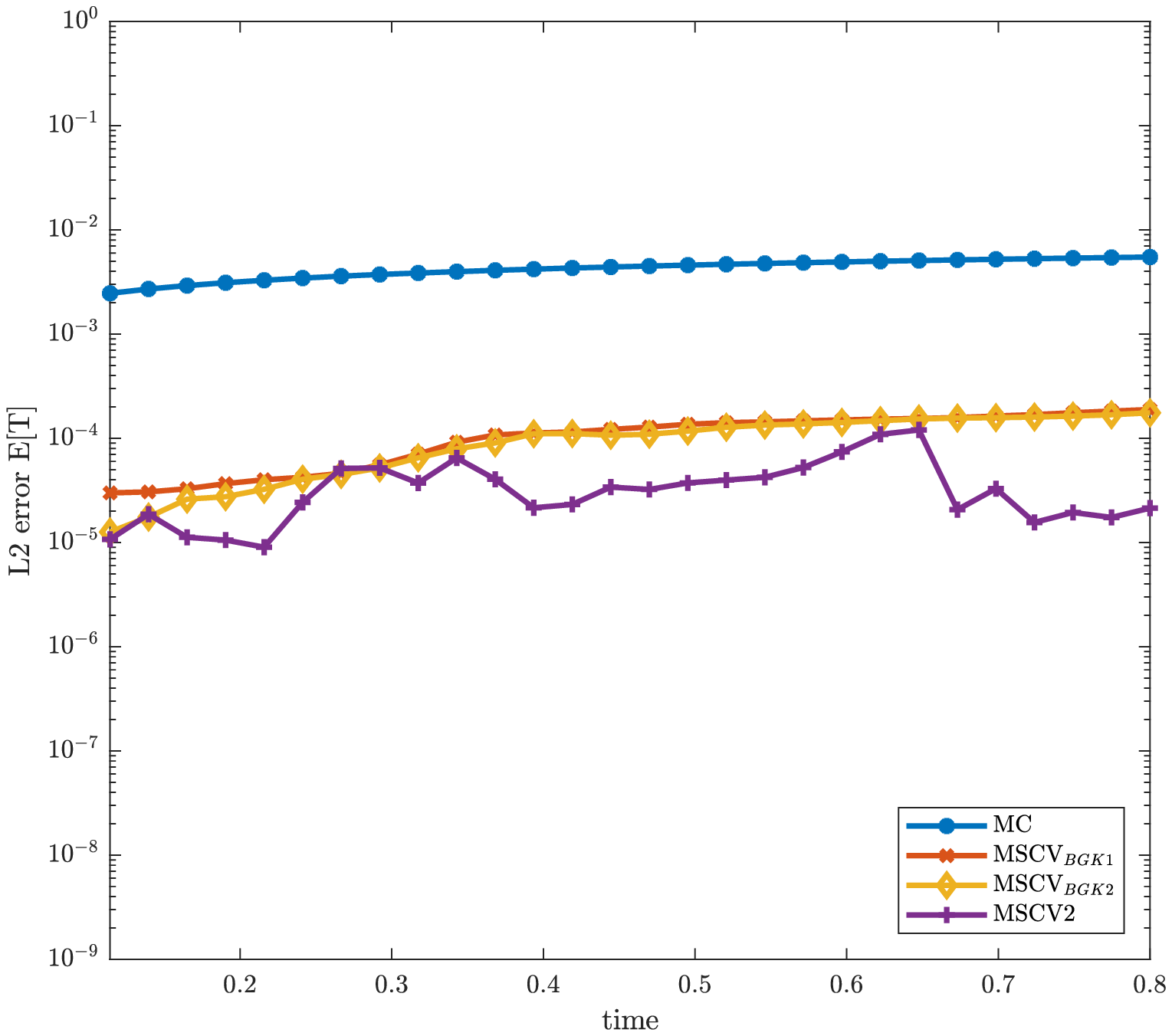}
		\includegraphics[width=.43\textwidth]{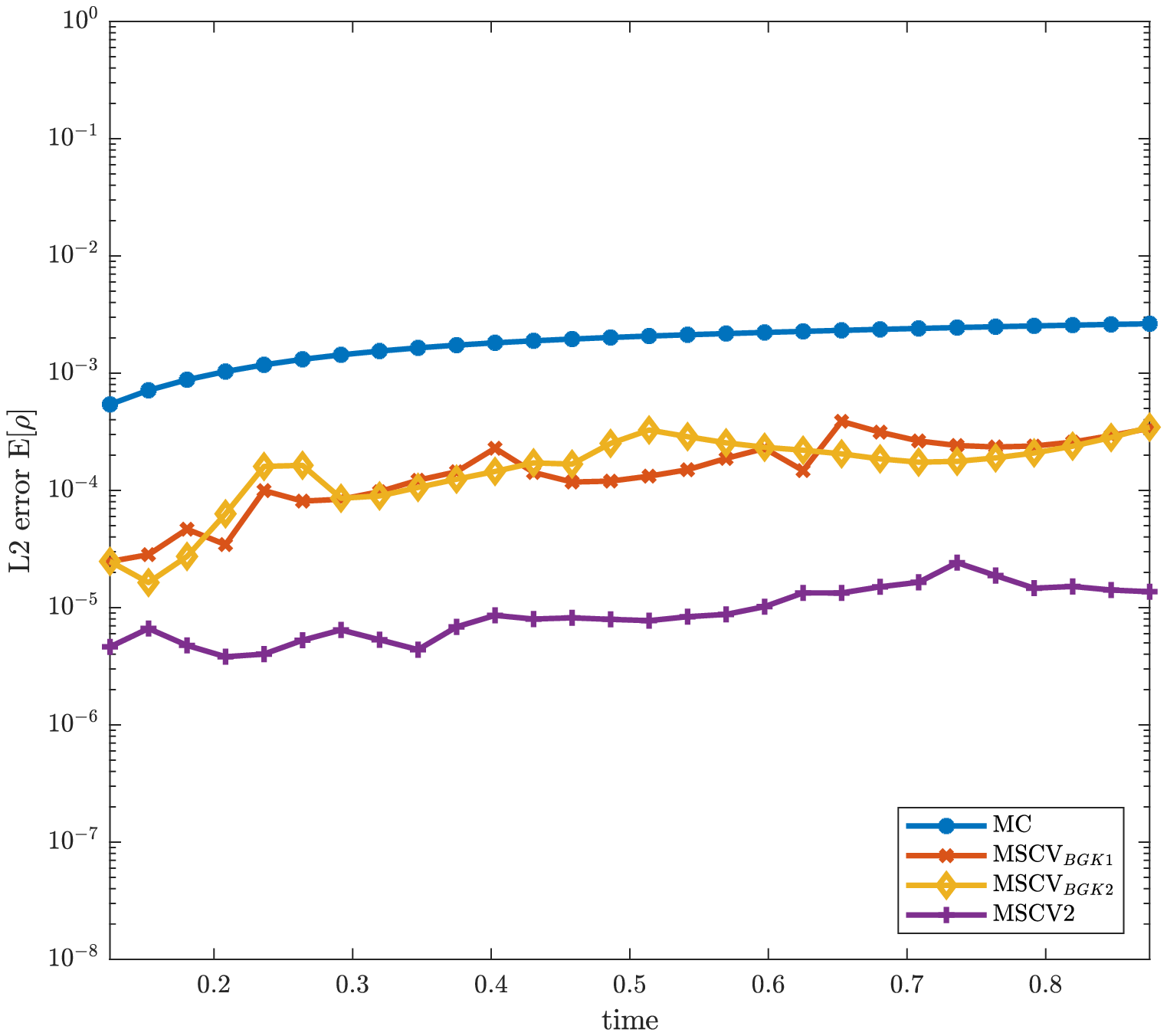}\\
		\includegraphics[width=.43\textwidth]{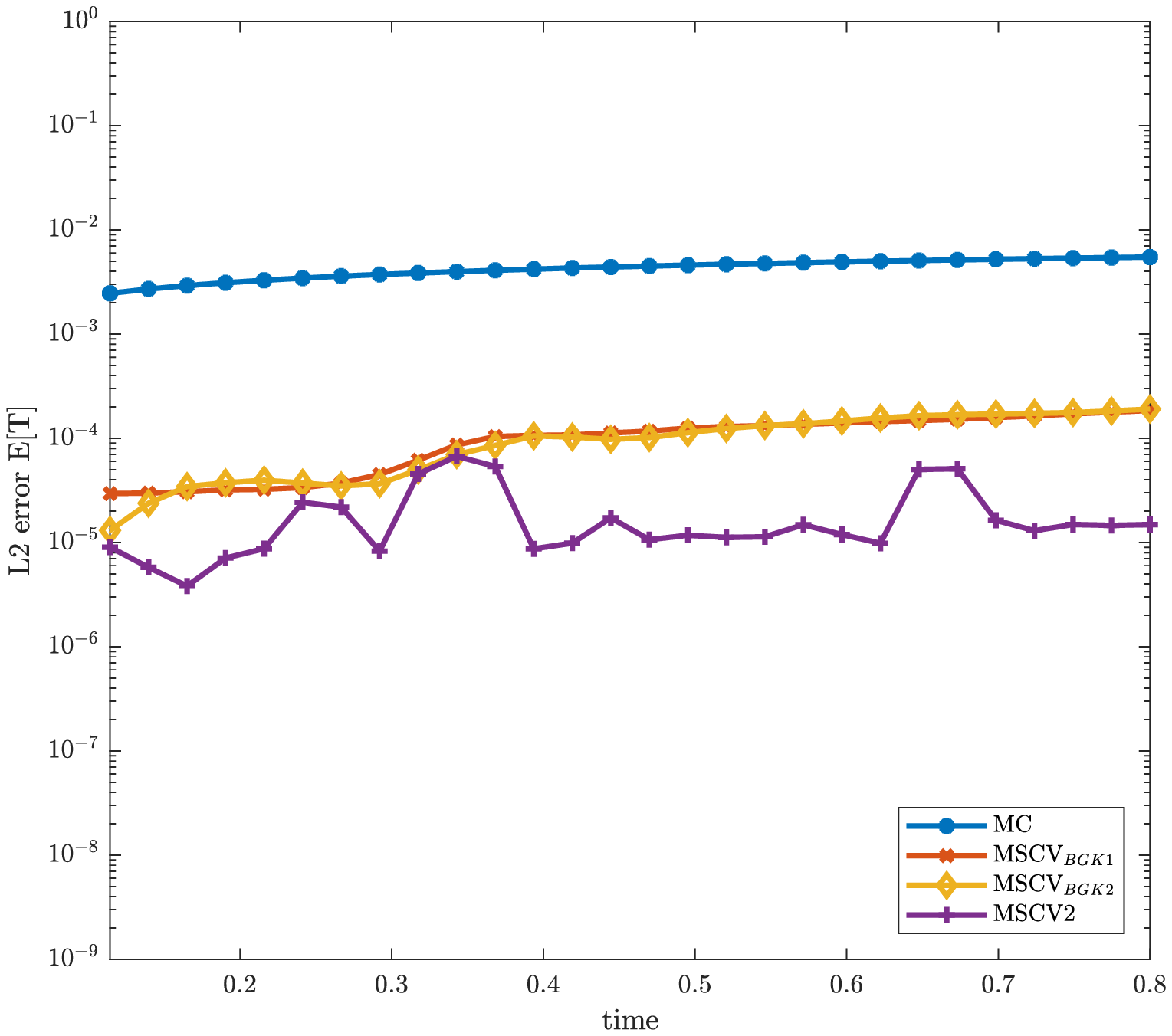}
		\includegraphics[width=.43\textwidth]{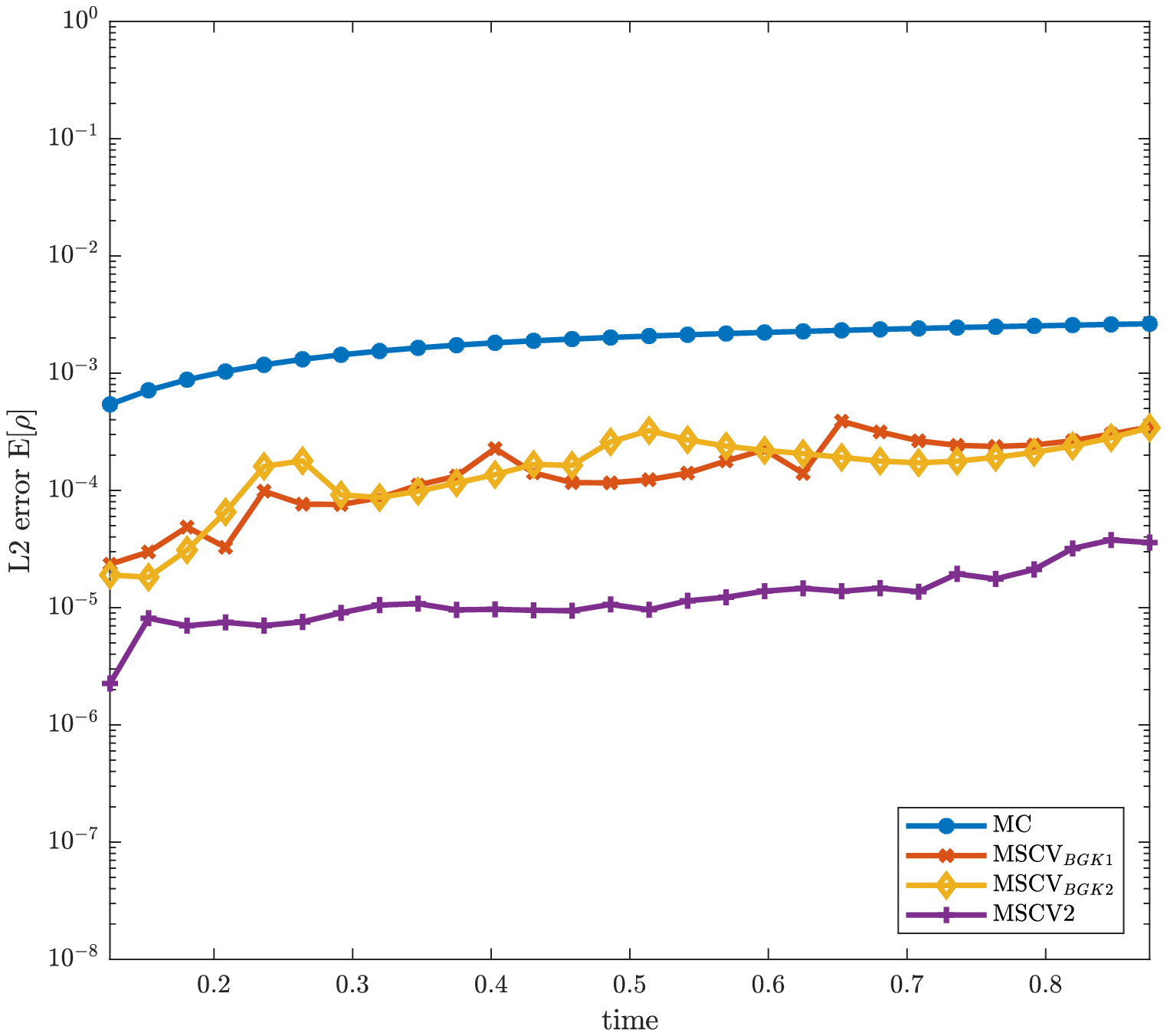}
		\caption{Test 3. Sudden heating problem with uncertainty in the boundary condition. $L_2$ norm of the error for the standard MC method, the MSCV method and the multiple control variate MSCV2 method for the expectation of the temperature (left) and for the density (right). The number of samples used for Boltzmann equation is $M=10$. Top: $M_E=10^3$ points. Middle: $M_E=5\times 10^3$ points. Bottom: $M_E=10^5$ points.}\label{Figure17}
	\end{center}
\end{figure}

\section{Conclusions}
We introduced a novel class of multiple control variate methods based on a multi-scale strategy for uncertainty quantification of kinetic equations and related problems. The approach generalizes the multi-scale control variate methodology recently introduced in \cite{DPms} to the case of multiple control variates. We discussed two different techniques, accordingly to the assumption that the multiple control variates possess a hierarchical multi-scale structure or not. We give theoretical and numerical evidence that these methods can improve the statistical estimates obtained with a single control variate in terms of ratio between accuracy and computational cost. Numerical comparison with other statistical estimators in the case of two control variates are reported for the Boltzmann equation. Finally, we point out that the present results, applied in a multi-level Monte Carlo setting, would naturally lead to optimal multi-level control variates. The latter application represents an interesting direction of research that will be explored in the nearby future.

%%%%%%%%%%%%%%%%%%%%%%%

%%%%%%%%%%%%%%%%%%%%%%%
%%%%%%%%%%%%%%%%%%%%%%%% referenc.tex %%%%%%%%%%%%%%%%%%%%%%%%%%%%%%
% sample references
% %
% Use this file as a template for your own input.
%
%%%%%%%%%%%%%%%%%%%%%%%% Springer-Verlag %%%%%%%%%%%%%%%%%%%%%%%%%%
%
% BibTeX users please use
% \bibliographystyle{}
% \bibliography{}
%

\end{document}